\documentclass[12pt,a4paper,reqno]{article}
\usepackage[utf8]{inputenc}
\usepackage{fontenc}
\usepackage{amsmath}
\usepackage{graphicx}
\usepackage{amssymb}
\usepackage{mathrsfs}  
\usepackage{amsfonts}
\usepackage{amsthm}
\usepackage[left=2.5cm, right=2.5cm, top=3cm, bottom=3.5cm]{geometry}
\usepackage[all]{xy}
\usepackage{appendix}
\usepackage[round]{natbib}
\usepackage[colorlinks=true,linkcolor=blue,citecolor=blue,bookmarks=true]{hyperref}
\usepackage{titlesec}
\usepackage{braket}
\usepackage{tikz-cd}
\usepackage{authblk}
\usepackage{subfig}

\DeclareMathAlphabet{\mathcal}{OMS}{cmsy}{m}{n}


\theoremstyle{plain}
\newtheorem{theorem}{Theorem}[section]
\newtheorem{lemma}[theorem]{Lemma}
\newtheorem{proposition}[theorem]{Proposition}
\newtheorem{corollary}[theorem]{Corollary}

\newtheorem*{theorem*}{Theorem}

\theoremstyle{definition}
\newtheorem{definition}{Definition}[section]

\theoremstyle{remark}
\newtheorem{remark}{Remark}[section]

\newcommand{\build}[3]{\mathrel{\mathop{\kern 0pt#1}\limits_{#2}^{#3}}}

\providecommand{\keywords}[1]
{
  \small
  \textbf{\textit{Keywords---}} #1
  \normalsize
}

\def\SO{{\mathrm{SO}}}

\def\C{{\mathbb C}}

\def\R{\mathbb{R}}
\def\Hom{\mathrm{Hom}}

\def\geq{\geqslant}
\def\leq{\leqslant}
\def\Id{\mathrm{Id}}
\def\E{\mathbb{E}}

\def\d{\mathrm{d}}
\def\e{\mathrm{e}}

\def\Var{\mathbb{V}\mathrm{ar}}

\def\epsilon{\varepsilon}

\def\bar{\overline}

\def\vol{\mathrm{vol}}
\def\eq{\mathrm{eq}}

\def\i{\mathrm{i}}
\def\Sbb{\mathbb{S}}

\def\ddc{\d\d^c}

\def\ddc{\d\d^c}
\def\M{{\mathcal{M}}}

\title{Monte Carlo on compact complex manifolds\\ using Bergman kernels}
\author[1]{Thibaut Lemoine\footnote{Corresponding author: \href{mailto:thibaut.lemoine@math.unistra.fr}{thibaut.lemoine@math.unistra.fr}}}
\author[2]{Rémi Bardenet}
\affil[1]{Institut de Recherche Math\'ematique Avanc\'ee, UMR 7501, Universit\'e de Strasbourg, 7, rue Ren\'e Descartes, 67084 Strasbourg, France}
\affil[2]{Université de Lille, CNRS, Centrale Lille\\
UMR 9189 -- CRIStAL, F-59000 Lille, France}
\date{}

\begin{document}

\maketitle

\begin{abstract}
In this paper, we propose a new randomized method for numerical integration on a compact complex manifold with respect to a continuous volume form. 
Taking for quadrature nodes a suitable determinantal point process, we build an unbiased Monte Carlo estimator of the integral of any $\mathscr{C}^1$ function, and show that the estimator satisfies a central limit theorem, with a faster rate than under independent sampling. 
In particular, seeing a complex manifold of dimension $d$ as a real manifold of dimension $d_\R=2d$, the mean squared error for $N$ quadrature nodes decays as $N^{-1-2/d_{\mathbb{R}}}$; this is faster than previous DPP-based quadratures and reaches the optimal worst-case rate investigated by \cite{Bak} in Euclidean spaces. 
The determinantal point process we use is characterized by its kernel, which is the Bergman kernel of a holomorphic Hermitian line bundle, and we build heavily on the work of Berman that led to the central limit theorem in \citep{Ber7}. We provide numerical illustrations for the Riemann sphere.
\end{abstract}


\tableofcontents

\section{Introduction}

Numerical integration, also known as quadrature, is the task of approximating an integral by a weighted sum of evaluations of the integrand. One usually distinguishes between Monte Carlo methods (MC; \citealp{Rob07}), where the nodes are random, and quasi-Monte Carlo methods (QMC; \citealp{DiPi10}), where they are deterministic. The two paradigms come with different notions of performance: for deterministic rules, one typically controls a worst-case error over a class of functions, whereas for randomized rules one usually derives concentration bounds or a central limit theorem for a fixed integrand as the number of nodes tends to infinity.

In recent years, several approaches have aimed at combining the flexibility of Monte Carlo sampling with the improved convergence rates usually associated with structured point sets.
This includes randomized QMC \citep{Owe97,Owe08}, kernel herding \citep{ChWeSm10,LiWa16}, or quadrature rules based on determinantal point processes (DPPs) \citep{BH,MCA,GaBaVa19b,BeBaCh19,Bel20,Bel21}. 
DPPs are random point processes with tractable correlation functions and a built-in repulsion between points; this negative dependence makes them natural candidates for structured Monte Carlo integration.

Most of the existing results on DPP-based quadrature concern Euclidean spaces. In this paper, we investigate instead the integration of functions over a compact complex manifold. Our quadrature nodes are sampled from a DPP canonically associated with a holomorphic line bundle over the manifold, whose kernel is the Bergman kernel. 
In particular, the Bergman kernel is projective with rank $N_k<\infty$, and the DPP thus puts all of its mass on configurations of $N_k$ points.
This geometric setting leads to two features that are, in our view, the main take-home messages of the paper.

On the one hand, the resulting Monte Carlo estimator is unbiased and satisfies a fast central limit theorem. If $\M$ has complex dimension $d$, hence real dimension $d_\R=2d$, then the fluctuations of the estimator are of order $N_k^{-1/2-1/d_\R}$.
Equivalently, the mean square error of our estimator is of order $N_k^{-1-2/d_\R}$. 
This matches the optimal worst-case rate established by \cite{Bak} for randomized integration of $\mathscr C^1$ functions in Euclidean spaces; see also \cite[Theorem 3]{Nov16}. On the other hand, the same DPP enjoys a \emph{universality} property: after the appropriate inverse-Bergman reweighting, it can be used to estimate integrals with respect to a whole family of measures absolutely continuous with respect to the original reference measure. The limiting variance is always given by the same Dirichlet-energy functional, applied to the corresponding reweighted integrand.

Our proofs strongly rely on the seminal work of \cite{Ber6,Ber7}, in particular on the central limit theorem for Bergman linear statistics proved in \cite{Ber7}. At a high level, the present paper can be read as the Monte Carlo counterpart of Berman's result: we show how to pass from a central limit theorem for linear statistics to an
unbiased quadrature rule. It can also be viewed as a complex-geometric analogue, and an improvement in the rate, of the DPP-based quadrature results of \cite{BH}. Finally, in the particular case of the sphere, our work is complementary to \cite{Ber8}, where the spherical ensemble is studied from a worst-case QMC viewpoint.

\subsection{Main result}

We now describe the setting and state our main result. The reader unfamiliar with the language of complex geometry is referred to Section~\ref{sec:DPP} and Appendix~\ref{appendix} for a more detailed introduction.

Let $L$ be a holomorphic line bundle over a compact complex manifold $\M$ of dimension $d$. If $h$ is an Hermitian metric on $L$ with local weight $\phi$, we shall denote by $\langle \cdot,\cdot\rangle_\phi$ and $|\cdot|_\phi$ the associated inner product and norm on each fiber. If $\mu$ is a finite measure on $\M$, then the pair $(\phi,\mu)$ induces an inner product on the space $H^0(M,L)$ of holomorphic sections of $L$:
\[
\langle s^{(1)},s^{(2)}\rangle_{(\phi,\mu)}=\int_\M h_x\left(s^{(1)}(x),s^{(2)}(x)\right)\d\mu(x).
\]
In the so-called \emph{semiclassical} setting, one replaces $L$ by $L^k=L^{\otimes k}$ and lets $k\to\infty$. The corresponding Hilbert space
\[
H_k=\big(H^0(M,L^k),\langle\cdot,\cdot\rangle_{(k\phi,\mu)}\big)
\]
is finite-dimensional; we denote its dimension by $N_k$.

The orthogonal projection $L^2(M,L^k)\to H^0(M,L^k)$ admits a reproducing kernel, denoted by $B_{(k\phi,\mu)}$ or simply $B_{(k\phi,\mu)}$, called the Bergman kernel. If $(s_i)_{1\leq i\leq N_k}$ is an orthonormal basis of $H_k$, then
\[
B_{(k\phi,\mu)}(x,y)=\sum_{i=1}^{N_k} s_i(x)\otimes s_i(y).
\]
On the diagonal, the Bergman kernel identifies with the function
\[
\mathcal B_{(k\phi,\mu)}(x)=\sum_{i=1}^{N_k} h_x^k(s_i(x),s_i(x)).
\]
The corresponding Bergman measure is defined by
\[
d\mu_{k\phi}(x)=\frac{1}{N_k}\mathcal B_{(k\phi,\mu)}(x)\d\mu(x).
\]

Under suitable assumptions on $(\phi,\mu)$, the Bergman measures converge towards an equilibrium measure $\mu_{\eq}^\phi$. In the smooth positive-curvature regime relevant for our main theorem, one can define the \emph{equilibrium weight}
\begin{equation}\label{eq:equil_weight}
w_{\mu}^\phi(x)=\frac{d\mu}{d\mu_\eq^\phi}(x).
\end{equation}
Beside this deterministic object, one can consider the random $N_k$-tuple $(X_1,\dots,X_{N_k})$ on $\M$ whose joint density with respect to $\mu^{\otimes N_k}$ is proportional to $|\det(s_i(x_j))|_{k\phi}^2,$ where $(s_i)$ is any orthonormal basis of $H_k$. This defines a determinantal point process with kernel $B_{(k\phi,\mu)}$, which we call the Bergman ensemble associated with the weighted measure $(k\phi,\mu)$.

A fundamental result of Berman states that the empirical average of a test function under this DPP satisfies a central limit theorem. For a positive weight $\phi$ and a function $f$, define the \emph{Dirichlet energy}
\begin{equation}\label{eq:Dirichlet}
\mathcal{E}_\phi(f):=V^{-1+1/d}\Vert\d f\Vert_{\ddc\phi}^2=V^{-1+1/d}\int_\M \vert\nabla f\vert_{\ddc\phi}^2\frac{(\ddc\phi)^d}{d!},\qquad V:=\int_\M\frac{(\ddc\phi)^d}{d!}.
\end{equation}

\begin{theorem}[Theorem 1.5 in \citealp{Ber7}]\label{thm:Berman}
Let $L$ be a holomorphic line bundle over a compact complex manifold $\M$. Let $(\phi,\mu)$ be a strongly regular weighted measure and $\mu_{\eq}$ be the associated Monge--Amp\`ere measure. For any $k\in\mathbb N^\ast$, let
$(X_1,\dots,X_{N_k})$ be a DPP with kernel $B_{(k\phi,\mu)}$. For any Lipschitz continuous $f:\M\to\mathbb R$ with compact support included in the weak bulk\footnote{See \citep[Theorem 3.4]{Ber} and \citep[Theorem 3.1]{Ber7} for a definition of the weak bulk.},
\[
N_k^{\frac12+\frac{1}{2d}}\left(\frac{1}{N_k}\sum_{i=1}^{N_k}f(X_i)-\mathbb E\left[\frac{1}{N_k}\sum_{i=1}^{N_k}f(X_i)\right]\right)\overset{(d)}{\longrightarrow}\mathcal N\left(0,\frac12\mathcal{E}_\phi(f)\right).
\]
\end{theorem}

As it stands, this theorem is not yet a quadrature result. Indeed,
\[
\mathbb E\left[\frac{1}{N_k}\sum_{i=1}^{N_k} f(X_i)\right]=\frac{1}{N_k}\int_\M f(x)B_{(k\phi,\mu)}(x,x)\d\mu(x),
\]
which depends on $k$ and converges to $\int_\M f\d\mu_{\eq}$, not to the target integral $\int_\M f\d\mu$. To obtain an unbiased Monte Carlo estimator of $\int_\M f\d\mu$, one has to compensate for the Bergman density on the diagonal. This leads to our main result.

\begin{theorem}\label{thm:MC_main}
Let $L$ be a positive holomorphic line bundle over a compact complex manifold $\M$, and let $(\phi,\mu)$ be a weighted measure such that $h=e^{-\phi}$ is a smooth metric with $\ddc\phi>0$, and $\mu$ is a probability measure defined by $\d\mu=\rho\frac{(\ddc\phi)^d}{d!},$ with $\rho\in \mathscr C^\infty(\M)$, $\rho>0$. Let $F$ be a line bundle endowed with a smooth local weight $\psi$. For any $k\in\mathbb N^\ast$, let $(X_1,\dots,X_{N_k})$ be a DPP with kernel $B_{(k\phi+\psi,\mu)}$. Then, for any $f\in \mathscr C^1(\M)$,
\begin{equation}
  \label{e:statement}
  N_k^{-\frac12+\frac{1}{2d}}\left(\sum_{i=1}^{N_k}\frac{f(X_i)}{B_{(k\phi+\psi,\mu)}(X_i,X_i)}-\int_\M f(x)\d\mu(x)\right)\overset{(d)}{\longrightarrow}\mathcal N\left(0,\frac12\mathcal{E}_\phi(fw_\mu^\phi)\right),
\end{equation}
where $w_\mu^\phi$ is defined in \eqref{eq:equil_weight}.
\end{theorem}

Several comments are in order.
First, the theorem shows that the error of the estimator in \eqref{e:statement} has size $N_k^{-1/2-1/(2d)}=N_k^{-1/2-1/d_\R}$, where $d_\R=2d$ is the dimension of $\mathcal{M}$ as a \emph{real} manifold.
In particular, the mean square error of the estimator decays as $N_k^{-1-1/d}=N_k^{-1-2/d_\R}.$ This is precisely the optimal worst-case rate obtained by \cite{Bak} in Euclidean spaces for randomized integration of $\mathscr C^1$ functions.
Second, at first sight, Theorem~\ref{thm:MC_main} might look like a direct consequence of Berman's Theorem~\ref{thm:Berman} with a $k$-dependent test function. Indeed, if one defines
\[
f_k(x)=f(x)\frac{N_k}{B_{(k\phi+\psi,\mu)}(x,x)},
\]
then the estimator in Theorem~\ref{thm:MC_main} can be rewritten as
\[
\frac{1}{N_k}\sum_{i=1}^{N_k} f_k(X_i),
\]
and one might hope to apply the central limit theorem in Theorem~\ref{thm:Berman} to $f_k$.
The difficulty is that Berman's result is formulated for a fixed test function, while the sequence $(f_k)$ depends on the Bergman kernel itself. 
Controlling the asymptotic behavior of these weighted observables therefore requires additional input. 
Our main technical contribution is precisely to extend Berman's asymptotic analysis to this setting. 
Under the stronger assumptions of Theorem~\ref{thm:MC_main}, the inverse diagonal Bergman density converges in $\mathscr C^1$ to the equilibrium weight $w_{\mu}^\phi$, and we prove a central limit theorem for weighted Bergman linear statistics whose test functions vary with $k$. This is the key step that turns Berman's linear-statistics CLT into an unbiased quadrature theorem.
Our third and last comment is the topic of the next subsection.

\subsection{Importance sampling and universality}

The fact that Theorem~\ref{thm:MC_main} holds with $B_{(k\phi+\psi,\mu)}$ in a way that the target integral and the asymptotic variance do \emph{not} depend on $\psi$ can be translated to an importance sampling result.

\begin{corollary}\label{cor:universality}
Let $L$ be a holomorphic line bundle over a compact complex manifold $\M$, and let $(\phi,\mu)$ satisfy the assumptions of Theorem~\ref{thm:MC_main}. For any smooth $\psi:\M\to\mathbb R$ such that $e^{-\psi}\mu$ is a probability measure, if $(X_1,\dots,X_{N_k})$ is a DPP with kernel $B_{(k\phi,e^{-\psi}\mu)}$, then for any $f\in\mathscr C^1(\M)$,
\begin{equation}
  \label{eq:universality1}
\sqrt{N_k^{1+1/d}}
\left(\sum_{i=1}^{N_k}\frac{f(X_i)e^{\psi(X_i)}}{B_{(k\phi,e^{-\psi}\mu)}(X_i,X_i)}-\int_\M f(x)\d\mu(x)\right)\overset{(d)}{\longrightarrow}\mathcal N\left(0,\frac12\mathcal{E}_\phi(fw_{\mu}^\phi)\right).
\end{equation}
\end{corollary}

Corollary~\ref{cor:universality} can be interpreted as follows: assume we want to estimate $\int_\M f(x)\d\mu(x)$ using the Bergman ensemble with kernel $B_{(k\phi,\mu)}$, but that the Bergman ensemble with kernel $B_{(k\phi,e^{-\psi}\mu)}$ is easier to sample. 
We can run our Monte Carlo procedure with this second Bergman ensemble and obtain an unbiased estimator with the \emph{same} asymptotic variance as with the first ensemble. 
In particular, the asymptotic variance does not depend on the density $e^{-\psi}$. 
On the one hand, this should be compared with classical importance sampling using independent variables, where changing the proposal distribution typically changes the asymptotic variance \citep{RoCa04}. 
On the other hand, the fact that the fast CLT is maintained relates to a universality phenomenon that was observed for multivariate orthogonal polynomial ensembles under the Nevai condition in \cite[Theorem 2.3 and Remark 4]{BH}; in our setting, universality appears in a geometric form through the invariance of the underlying Bergman ensemble under a suitable simultaneous change of metric and reference measure.

There is a complementary, dual way to formulate this universality. 
Fix a smooth probability measure $\mu$ satisfying the assumptions of Theorem~\ref{thm:MC_main}, and sample the Bergman ensemble associated with $(k\phi,\mu)$. Then this single Bergman ensemble 
yields fast unbiased quadrature rules for the whole smooth class $\{e^\psi\mu:\psi\in \mathscr C^\infty(M),\ \int_\M e^\psi\mu=1\}$. The fluctuation scale is universal over this class, while the limiting variance changes only through the target weight $e^\psi w_\mu^\phi.$

\subsection{Relation to previous work}

Our initial motivation was to generalize the work of \cite{BH}, where integration on the hypercube is performed using multivariate orthogonal polynomial ensembles. Their theorem yields a fast central limit theorem in a real Euclidean setting. The present paper shows that Bergman ensembles play
an analogous role on compact complex manifolds, but with a better rate when the complex structure is taken into account: if $\M$ has complex dimension $d$, then the real dimension is $d_\R=2d$, and our error rate is of order $N^{-1/2-1/d_\R}$.

Our work is also more generally related to recent progress on numerical integration on manifolds. On the deterministic side, QMC constructions on compact Riemannian manifolds were investigated for instance in \cite{BCCGST,BDE,Ber8}. On the randomized side, Markov chain Monte Carlo methods on manifolds embedded in Euclidean spaces were studied for example by \cite{GC,DHS,ZHG,EGO}. These methods address different questions and rely on different assumptions, but they
provide useful benchmarks. Compared with MCMC-type methods, our DPP-based quadrature achieves a faster fluctuation scale when measured in number of function evaluations. Compared with QMC-type methods, it yields an unbiased randomized rule together with an explicit Gaussian fluctuation theory.

Finally, in the special case of the sphere, a few DPPs have been investigated.  The so-called harmonic ensemble, for instance, corresponds to the projection onto spherical harmonics of bounded degree. Starting with the seminal work of \cite{Joh97} in real dimension $1$, central limit theorems have been obtained for the harmonic ensemble, with the same convergence rate as in \cite{BaHa20}; see \cite{LeMaOr24} for an elegant analytic proof and a survey of previous results. Alternately, the spherical ensemble with $d=1$ was introduced in \cite{Kri09}, and a central limit theorem with the same rate as our Theorem~\ref{thm:MC_main} was first obtained by \cite{RiVi07}. We also refer to \cite{LeMaOr24} for an alternative proof and references. \cite{Ber8} recently studied the spherical ensemble from the perspective of worst-case quadrature in Sobolev spaces. Our Theorem~\ref{thm:MC_main} may be viewed as the corresponding Monte Carlo statement for smooth enough fixed integrands, and Section~\ref{sec:sphere} provides numerical illustrations in this case.

Let us also mention the works of Beltr\'an and Etayo on spherical-type determinantal processes in higher dimension. In dimension one, the Bergman ensemble on $\mathbb{CP}^1\simeq S^2$ is precisely the classical spherical ensemble. \cite{BE18} extended the spherical ensemble to a projective ensemble on $\mathbb{CP}^d$, which can be seen as a particular case of Bergman ensemble, and \cite{BE19} constructed a generalized spherical ensemble on $S^{2d}$. These papers develop a clear relationship between Bergman ensembles and spherical ensembles in higher dimensions.

\subsection{Outline of the paper}

The rest of the paper is organized as follows. In Section~\ref{sec:DPP}, we introduce the Bergman ensembles that will be used as quadrature nodes and recall the Bergman-kernel asymptotics needed later. Section~\ref{sec:CLT} contains a central limit theorem for weighted Bergman linear statistics, which is the main technical ingredient of the paper. In Section~\ref{sec:proofs}, we deduce Theorem~\ref{thm:MC_main} and Corollary~\ref{cor:universality} from this result.

In Section~\ref{sec:sphere}, we make the construction explicit on the Riemann sphere, where the Bergman kernel and the associated DPP can be written in closed form, and we compare our estimator with several competitors by numerical experiments. Finally, Section~\ref{sec:discussion} discusses limitations of the method and possible directions for future work.

\section{Bergman ensembles}\label{sec:DPP}

In this section, we will introduce the point processes involved in our framework: the Bergman ensembles. They live on a complex manifold $\M$, and they are related to the reproducing kernel of a Hilbert space of holomorphic sections of a line bundle over $\M$, called the Bergman kernel. 
We will assume basic knowledge of complex manifolds and line bundles, although we will spend some time describing the Bergman kernel later; however, since we want to ensure that the readers who are rather familiar with numerical integration than with complex geometry, we recall in Appendix~\ref{appendix} the minimal definitions required to understand properly the objects in play here.

\subsection{Definitions}\label{sec:DPP_def}

Let $\M$ be a compact complex manifold of dimension $d$ endowed with a Borel probability measure $\mu$ induced by a continuous positive volume form, $L$ be a holomorphic line bundle over $\M$ with Hermitian metric $h$, represented by a local weight function $\phi$. Set $N=\dim H^0(\M,L)$, and denote by $\mathcal{P}_{(\phi,\mu)}$ the probability measure on $\M^{N}$ defined by
\begin{equation}
  \label{eq:P_phi}
  \d\mathcal{P}_{(\phi,\mu)}(x_1,\ldots,x_N)=\frac{1}{N!}\vert\det(s_i(x_j))\vert_{\phi}^2\d\mu^{\otimes N}(x_1,\ldots,x_N),
\end{equation}
where $(s_i)$ is an orthonormal basis of $(H^0(\M,L),\langle\cdot,\cdot\rangle_{(\phi,\mu)})$.
The fact that $\mathcal{P}_{(\phi,\mu)}$ has unit mass is a standard application of the generalized Cauchy--Binet identity \citep[Proposition 2.10]{Joh06}. We will denote by $\E_{(\phi,\mu)}$ the expectation with respect to $\mathcal{P}_{(\phi,\mu)}$, i.e., for any bounded measurable $F:\M^N\to\C$,
\[
\E_{(\phi,\mu)}[F(X_1,\ldots,X_N)] = \int_{\M^N} F \d\mathcal{P}_{(\phi,\mu)}.
\]
\begin{definition}
  \label{def:bergman_ensemble}
  The \emph{Bergman ensemble} for the weighted measure $(\phi,\mu)$ is the point process $\sum_{i=1}^N\delta_{X_i},$ where $(X_1,\ldots,X_N)$ is a family of random variables on $\M$ with distribution $\mathcal{P}_{(\phi,\mu)}$ in \eqref{eq:P_phi}.
\end{definition}

Recall that a (simple) point process on $\M$ is a random configuration on $\M$, or equivalently the counting measure of this configuration. 
In particular, given a family $(X_1,\ldots,X_N)$ of almost-surely distinct random variables on $\M$, the random measure $\sum_{i=1}^N\delta_{X_i}$ almost surely has distinct atoms, and thus defines a simple point process. 
The $n$-point correlation function $\rho_n:\M^n\to\R$ of a simple point process, when it exists, is characterized by the property that, for any bounded measurable function $f:\M^n\to\R$,
\[
\E\left[\sum_{i_1\neq\cdots\neq i_n} f(X_{i_1},\ldots,X_{i_n})\right] = \int_{\M^n}f(x_1,\ldots,x_n)\rho_n(x_1,\ldots,x_n)\d\mu^{\otimes n}(x_1,\ldots,x_n).
\]
A simple point process is said to be \emph{determinantal} with kernel $K:\M\times\M\to\C$ if all its correlation functions can be expressed as determinants of a matrix whose coefficients are given by evaluations of the kernel:
\begin{equation}
\rho_n(x_1,\ldots,x_n)=\det(K(x_i,x_j))_{1\leq i,j\leq n}.
\end{equation}
It shows a structural repulsion of the points in the configuration, because the correlation vanishes trivially when there are $i\neq j$ such that $x_i=x_j$, thanks to the alternate property of the determinant. We will see in a moment that the Bergman ensemble is determinantal, but we will first spend the next subsection defining the underlying kernel.

\subsection{Bergman kernel}\label{sec:Bergman-kernel}

Let $L$ be a holomorphic line bundle over a compact complex manifold $\M$, endowed with an Hermitian metric $h$. Let $\mu$ be a continuous volume form on $\M$ and $\phi$ be the local weight corresponding to $h$. The \emph{Bergman kernel} $B_{(\phi,\mu)}$ of $L$ with respect to the weighted measure $(\phi,\mu)$ is the Schwartz kernel (cf. \citealp[Chapter 6]{LeF} for a detailed introduction of such kernels) of the orthogonal projection $L^2(\M,L) \longrightarrow H^0(\M,L)$. Namely, it is a section of $L\boxtimes\overline{L}\rightarrow \M\times \M$, and it can be written as
\begin{equation}\label{eq:unweighted_Bergman_kernel}
B_{(\phi,\mu)}(x,y) = \sum_{i=1}^Ns_i(x)\otimes\overline{s_i(y)},\ \forall x,y\in \M,
\end{equation}
where $(s_i)$ is an orthonormal basis of $H^0(\M,L)$ for the inner product $\langle\cdot,\cdot\rangle_{(\phi,\mu)}$. By construction, $B_{(\phi,\mu)}$ is the reproducing kernel of the Hilbert space $(H^0(\M,L),\langle\cdot,\cdot\rangle_{(\phi,\mu)})$, which means that
\begin{equation}\label{eq:reprod_Bergman}
\int_\M B_{(\phi,\mu)}(x,y)\cdot s(y)\d\mu(y) = s(x),\ \forall s\in H^0(\M,L),\ \forall x\in \M.
\end{equation}
The dot in \eqref{eq:reprod_Bergman} represents the contraction between the Bergman kernel and the section $s$ induced by \eqref{eq:contraction}, so that the left-hand side of \eqref{eq:reprod_Bergman} is the decomposition of $s$ onto the orthogonal basis $(s_i)$; see \citep[Lemma 6.3.2]{LeF} for more details.

As we shall see later, the correlation functions of our point processes will be expressed as determinants of the Bergman kernel. Let us stress that such a determinant is not canonically defined: for instance, we have for $x,y\in\M$, at least formally,
\begin{equation}\label{eq:wrong_det}
\left\vert\begin{matrix}
B_{(\phi,\mu)}(x,x) & B_{(\phi,\mu)}(x,y)\\
B_{(\phi,\mu)}(y, x) & B_{(\phi,\mu)}(y,y)
\end{matrix}\right\vert = B_{(\phi,\mu)}(x,x)\otimes B_{(\phi,\mu)}(y,y) - B_{(\phi,\mu)}(x,y)\otimes B_{(\phi,\mu)}(y,x).
\end{equation}
However, this equality does not really make sense because we are adding elements of two different vector spaces: one is a section of $L_x\otimes\overline{L}_x\otimes L_y\otimes\overline{L}_y$ and the other is a section of $L_x\otimes\overline{L}_y\otimes L_y\otimes\overline{L}_x$. We circumvent this difficulty by using contractions on tensor products of each fiber and its dual, thanks to the following isomorphism of vector spaces, for any finite-dimensional vector space $E$,
\[
\left\lbrace\begin{array}{ccc}
L_x\otimes E\otimes\overline{L_x} & \build{\longrightarrow}{}{\sim} & E\\
u_x\otimes w\otimes\overline{v_x} & \longmapsto & h_x(u_x,v_x)w
\end{array}\right..
\]
In particular, the Bergman kernel on the diagonal can be identified with the function $\mathcal{B}_{(\phi,\mu)}:\M\to\C$ given by
\[
B_{(\phi,\mu)}(x,x)=\mathcal{B}_{(\phi,\mu)}(x): = \sum_{i=1}^N h_x(s_i(x),s_i(x)),\ \forall x\in \M.
\]
The reproducing property leads to the following characterization of the Bergman kernel on the diagonal, called \emph{extremal property},
\begin{equation}\label{eq:ext_prop_Bergman}
\mathcal B_{(\phi,\mu)}(x) = \sup\{\vert s(x)\vert_\phi^2:\ s\in H^0(\M,L), \Vert s\Vert_{(\phi,\mu)}^2\leq 1\}.
\end{equation}
The reader familiar with orthogonal polynomials will notice the similarity with the extremality property of the Christoffel function \citep[Chapter 1, Section 8]{Sim05}.
Moreover, using the contraction $\cdot$ in the sense of~\eqref{eq:contraction}, Equation \eqref{eq:wrong_det} becomes
\[
\left\vert\begin{matrix}
B_{(\phi,\mu)}(x,x) & B_{(\phi,\mu)}(x,y)\\
B_{(\phi,\mu)}(y, x) & B_{(\phi,\mu)}(y,y)
\end{matrix}\right\vert = B_{(\phi,\mu)}(x,x)B_{(\phi,\mu)}(y,y) - B_{(\phi,\mu)}(x,y)\cdot B_{(\phi,\mu)}(y,x).
\]
Following \citep{Lem}, this leads to the following convention for the definition of the determinant.
\begin{definition}\label{def:det_Bk}
The \emph{determinant} $\det(B_{(\phi,\mu)}(x_i,x_j))_{1\leq i,j\leq n}$ is defined by
\begin{equation}
  \label{e:def_determinants}
  \det(B_{(\phi,\mu)}(x_i,x_j))_{1\leq i,j\leq n} = \sum_{\sigma\in\mathfrak{S}_n} \varepsilon(\sigma) \sum_{i_1,\ldots,i_n=1}^N \prod_{j=1}^n h_{x_j}(s_{i_j}(x_j),s_{i_{\sigma^{-1}(j)}}(x_j)).
\end{equation}
\end{definition}

\begin{remark}\label{rmk:change_kernel}
By standard arguments, if $\psi:\M\to\R$ is a continuous function, then the metric $\widetilde{h}=he^{-\psi}$ yields a new Hermitian metric on $L$, and for any sections $s,t\in H^0(\M,L)$, we have the obvious identity
\[
\int_\M h_x(s(x),t(x))\d\mu(x)=\int_\M \widetilde{h}_x(s(x),t(x))e^{\psi(x)}d\mu(x),
\]
so that the $L^2$ structures on $H^0(\M,L)$ obtained from $(h,\mu)$ and $(\widetilde{h},e^{\psi}\mu)$ coincide. In particular, if we denote by $\phi$ the local weight of $h$, so that $h=e^{-\phi}$ and $\widetilde{h}=e^{-\phi-\psi}$, we have the following important identification between Bergman kernels: $B_{(\phi,\mu)}=B_{(\phi+\psi,e^{\psi}\mu)}$. However, the identifications of the kernels by contractions are \emph{not} the same: for instance, on the diagonal,
\[
\mathcal B_{(\phi+\psi,e^{\psi}\mu)}(x)=\sum_{i=1}^N h_x(s_i(x),s_i(x))e^{-\psi(x)}=\mathcal B_{(\phi,\mu)}(x)e^{-\psi(x)},\quad \forall x\in\M.
\]
This also entails that the determinants, as defined in \eqref{e:def_determinants}, do not coincide, but rather differ by a scaling factor:
\[
\det(B_{(\phi+\psi,e^{\psi}\mu)}(x_i,x_j))_{1\leq i,j\leq n}=\det(B_{(\phi,\mu)}(x_i,x_j))_{1\leq i,j\leq n}\prod_{i=1}^n e^{-\psi(x_i)}.
\]
\end{remark}

\begin{remark}
If we consider a small open subset $U\subset\M$ with local coordinates $z\in\C^d$ where $L$ is trivialized, each holomorphic section $s_i$ can be written $s_i(z)=f(z)e_U(z)$ where $f:\C^d\to\C$ is a holomorphic map and $e_U$ is a local frame, so that
\[
\vert s_i(z)\vert_{\phi}=\vert f_i(z)\vert \vert e_U(z)\vert=\vert f_i(z)\vert e^{-\frac12\phi(z)}.
\]
An important fact is that, although $s_i$ is a holomorphic section and $f_i$ is a holomorphic map, there is no reason for $\vert s_i(z)\vert_\phi$ to be holomorphic, because $z\mapsto e^{-\frac12\phi(z)}$ is not necessarily holomorphic. This is why, when considering local coordinates on $U$, it will sometimes be more convenient to work with the \emph{unweighted} Bergman kernel $K_{(\phi,\mu)}(z,w)=\sum_i f_i(z)\overline{f_i}(w)$, which is holomorphic in $z$ and antiholomorphic in $w$, compared to the weighted Bergman kernel 
$$ 
  B_{(\phi,\mu)}(z,w)=\sum_i s_i(z)\overline{s_i}(w)=K_{(\phi,\mu)}(z,w)e^{-\frac12\phi(z)}e^{-\frac12\phi(w)},
$$
which is not.

Unlike its name suggests, the unweighted kernel $K_{(\phi,\mu)}$ \emph{still depends} on the weighted measure $(\phi,\mu)$, because the functions $f_i$ are local representatives of the family $(s_i)$ of sections that are orthonormal \emph{with respect to the weighted measure}. In light of Remark~\ref{rmk:change_kernel}, the ambiguity of identifications of kernels $B_{(\phi+\psi,e^{\psi}\mu)}$ and $B_{(\phi,\mu)}$ when using contractions vanishes for unweighted kernels: in local coordinates, we still get the equality $K_{(\phi+\psi,e^{\psi}\mu)}(z,w)=K_{(\phi,\mu)}(z,w)$ because the Bergman kernels are associated with the same orthonormal basis of holomorphic sections. 
  This is why in some references, like \citep{Ber7}, the Bergman kernel is identified with the unweighted kernel at a global level, writing for instance $K_k(x,y)$ instead of $B_{(k\phi,\mu)}$ for the global section of $L^k\boxtimes \overline{L}^k$. 
  We prefer to keep the notation that involves the weighted measure to stress the dependence on both the reference measure and the Hermitian metric, although experts in complex geometry might find it unnecessarily cumbersome.
\end{remark}

\subsection{Bergman ensemble in the semiclassical setting}

The Bergman ensemble, as shown in \citep{Lem} for instance, is a \emph{determinantal point process} (DPP) with kernel $B_{(\phi,\mu)}$; recall that it means that its correlation functions exist and are given by
\begin{equation}
\rho_n(x_1,\ldots,x_n) = \det(B_{(\phi,\mu)}(x_i,x_j))_{1\leq i,j\leq n},\ \forall n\geq 1,\ \forall x_1,\ldots,x_n\in\M.
\end{equation}
Note that the determinant of the Bergman kernel is given in Definition \ref{def:det_Bk}. 
Let us remark that, in particular, the joint density of \eqref{eq:P_phi} can be rewritten in terms of the Bergman kernel, 
\[
\d\mathcal{P}_{(\phi,\mu)}(x_1,\ldots,x_N)=\frac{1}{N!}\det(B_{(\phi,\mu)}(x_i,x_j))_{1\leq i,j\leq N}\d\mu^{\otimes N}(x_1,\ldots,x_N).
\]
In the parlance of determinantal point processes \citep{HKPV06}, the Bergman ensemble is a projection DPP.
Another important comment is that, by Remark~\ref{rmk:change_kernel}, a DPP of kernel $B_{(\phi,\mu)}$ is also a DPP of kernel $B_{(\phi+\psi,e^{\psi}\mu)}$.
Indeed, although the determinant has a scaling factor due to the contractions, the factor exactly cancels out the one obtained in the change of reference measure, and the whole probability measure remains the same. 

Until now, the line bundle has been arbitrary, but in the sequel we shall focus on the case where $L$ is replaced by $L^k\otimes F$, for large $k$, with corresponding weight $k\phi+\psi$. It is sometimes called a \emph{semiclassical setting}, and it is in particular a standard framework in geometric quantization (see \citep{Cha03} for instance). Since $\M$ is compact, the space $H^0(\M,L^k\otimes F)$ has finite dimension $N_k$. We will also put a further assumption on the line bundle.

\begin{definition}
Let $L$ be a holomorphic line bundle over a compact complex manifold $\M$. $L$ is said to be
\emph{positive} or \emph{ample} if there exists a smooth Hermitian metric $h_\phi$ on $L$ such that $\ddc\phi>0$.
\end{definition} 
The assumption of a positive bundle, which is required in our Theorem~\ref{thm:MC_main}, is actually quite strong; in contrast, in Berman's Theorem~\ref{thm:Berman}, the line bundle is assumed to be \emph{big}, which is weaker -- and we shall not enter in detail in a precise comparison here. The reason we need the positivity assumption is that one of our key estimates only works in this setting.

\subsection{Convergence of the Bergman measures}

The first macroscopic estimation of the Bergman ensemble for the weighted measure $(k\phi+\psi,\mu)$ for large $k$ is given by the asymptotics of the \emph{Bergman measures} $\beta_k(x)\d\mu(x)=\frac{1}{N_k}B_{(k\phi+\psi,\mu)}(x,x)\d\mu(x)$. Indeed, as
\[
\E_{k\phi+\psi}\left[\frac{1}{N_k}\sum_{i=1}^{N_k} f(X_i)\right] = \int_\M f(x)\d\beta_k(x),
\]
the weak convergence in expectation of the empirical measures of the DPP is equivalent to the weak convergence of the Bergman measures. In the case of a compact complex manifold $\M$ endowed with a positive Hermitian line bundle $L$, the Bergman measures converge pointwise (hence weakly) to an equilibrium measure $\mu_\eq$ because of the diagonal expansion of the Bergman kernel. In the more general case studied by \cite{Ber7}, the (weak or pointwise) convergence of Bergman measures is not automatic.

The Bergman measure $\beta_k$ has density $\frac{1}{N_k}\mathcal B_{(k\phi+\psi,\mu)}$ with respect to $\d\mu$; we will actually control the convergence of the inverse of this density, for a reason that will appear later. Let $w_\mu^\phi$ be the equilibrium weight defined in~\eqref{eq:equil_weight}. The following lemma is an analog of a classical result about the Christoffel--Darboux kernel for orthogonal polynomials on the unit circle \citep[Theorem 2.15.1]{Sim} and on the real segment $[-1,1]$ \citep[Theorem 3.11.1]{Sim}. 

\begin{lemma}\label{lem:ratio_C1}
Let $\M$ be a compact complex manifold of dimension $d$, and let $L\to \M$ be a positive holomorphic line bundle endowed with a smooth Hermitian metric $h=e^{-\phi}$ such that $\ddc\phi>0$ on $\M$. Let $\mu=\rho\frac{\omega^d}{d!}$ with $\omega=\ddc\phi$ and $\rho\in \mathscr{C}^\infty(\M)$, $\rho>0$. Let $F\to M$ be a holomorphic line bundle endowed with a smooth Hermitian metric $e^{-\psi}$. For $k$ large enough, the ratio $r_k(x):=N_k/\mathcal B_{(k\phi+\psi,\mu)}(x)$ is well-defined, and we have
\begin{equation}\label{eq:Totik}
\|r_k-w_\mu^\phi\|_{\mathscr{C}^1(\M)}=O(k^{-1}).
\end{equation}
\end{lemma}

\begin{proof}
Set
\[
\omega:=\ddc\phi,\qquad \Omega:=\frac{\omega^d}{d!},\qquad V:=\int_\M \Omega .
\]
Thus $\d\mu=\rho\Omega$, with $\rho\in \mathscr C^\infty(\M)$ and $\rho>0$. We first reduce the reference measure to the K\"ahler volume form $\Omega$. Define a new Hermitian metric on $F$ by $\widehat h_F:=\rho h_F .$ Equivalently, if $h_F=e^{-\psi}$ locally, then
\[
\widehat h_F=e^{-\widehat\psi},\qquad \widehat\psi=\psi-\log\rho .
\]
Since $\rho$ is smooth and positive and $h_F$ is smooth, $\widehat h_F$ is a smooth Hermitian metric on $F$. For any $s,t\in H^0(\M,L^k\otimes F)$, we have
\[
\int_\M \langle s(x),t(x)\rangle_{h_L^k\otimes h_F}d\mu(x)=\int_\M \langle s(x),t(x)\rangle_{h_L^k\otimes h_F}\rho(x)\Omega(x)=\int_\M \langle s(x),t(x)\rangle_{h_L^k\otimes \widehat h_F}\Omega(x).
\]
It follows that the two Hilbert structures $\bigl(H^0(\M,L^k\otimes F), h_L^k\otimes h_F,\mu\bigr)$ and $\bigl(H^0(\M,L^k\otimes F), h_L^k\otimes \widehat h_F,\Omega\bigr)$ coincide.

Let $(s_{j,k})_{1\leq j\leq N_k}$ be an orthonormal basis for this common Hilbert structure. Then
\[
\mathcal B_k(x):=\mathcal B_{(k\phi+\psi,\mu)}(x)=\sum_{j=1}^{N_k}|s_{j,k}(x)|^2_{h_L^k\otimes h_F}.
\]
Let $\widehat{\mathcal B}_k(x)$ denote the scalar diagonal Bergman kernel computed with respect to the metric $h_L^k\otimes \widehat h_F$ and the reference volume form $\Omega$. With the same orthonormal basis,
\[
\widehat{\mathcal B}_k(x)=\sum_{j=1}^{N_k}|s_{j,k}(x)|^2_{h_L^k\otimes \widehat h_F}.
\]
Since $\widehat h_F=\rho h_F$, we get the pointwise identity
\[
\widehat{\mathcal B}_k(x)=\rho(x)\mathcal B_k(x).
\]
We now apply the smooth diagonal Bergman kernel expansion for a positive line bundle with a smooth auxiliary bundle. More precisely, by \cite[Theorem 4.1.1]{MaMa}, applied to the positive line bundle $(L,h_L)$, the auxiliary line bundle $(F,\widehat h_F)$, and the K\"ahler volume form $\Omega=\omega^d/d!$, one has, for every $m\geq 0$,
a $\mathscr C^m$-asymptotic expansion
\[
\widehat{\mathcal B}_k\sim k^d \widehat b_0+k^{d-1}\widehat b_1+k^{d-2}\widehat b_2+\cdots .
\]
With the present normalization of $\Omega$, the leading coefficient is $\widehat b_0\equiv 1.$ In particular, taking $m=1$ and keeping only the first term gives
\[
k^{-d}\widehat{\mathcal B}_k=1+O_{\mathscr C^1(M)}(k^{-1}).
\]
Using $\widehat{\mathcal B}_k=\rho \mathcal B_k$, we obtain
\[
k^{-d}\mathcal B_k= \rho^{-1}k^{-d}\widehat{\mathcal B}_k=\rho^{-1}+O_{\mathscr C^1(\M)}(k^{-1}).
\]
Set $b_k:=k^{-d}\mathcal B_k.$ Then
\[
b_k=\rho^{-1}+O_{\mathscr C^1(\M)}(k^{-1}).
\]

Next, by the definition of the diagonal Bergman kernel and by orthonormality,
\[
N_k=\sum_{j=1}^{N_k}\|s_{j,k}\|^2=\int_\M \mathcal B_kd\mu=\int_\M \widehat{\mathcal B}_k\Omega .
\]
Integrating the preceding Bergman kernel expansion gives
\begin{equation}\label{eq:volume}
N_k=k^d\int_\M \Omega+O(k^{d-1})=Vk^d+O(k^{d-1}).
\end{equation}
Equivalently, if $n_k:=k^{-d}N_k,$ then
\[
n_k=V+O(k^{-1}).
\]
Since $\rho^{-1}$ is positive on the compact manifold $\M$, there exists $c>0$ such that $\rho^{-1}\geq c$. The convergence
\[
b_k=\rho^{-1}+O_{\mathscr C^1(\M)}(k^{-1})
\]
therefore implies that, for $k$ large enough, $b_k\geq c/2$. Hence inversion is a smooth operation in a $\mathscr C^1$-neighbourhood of $b_k$, and
\[
b_k^{-1}=\rho+O_{\mathscr C^1(\M)}(k^{-1}).
\]
Therefore
\[
r_k=\frac{N_k}{\mathcal B_k}=\frac{k^{-d}N_k}{k^{-d}\mathcal B_k}=\frac{n_k}{b_k}=\bigl(V+O(k^{-1})\bigr)\bigl(\rho+O_{\mathscr C^1(\M)}(k^{-1})\bigr).
\]
Finally,
\[
r_k=V\rho+O_{\mathscr C^1(\M)}(k^{-1}).
\]

It remains to identify $V\rho$ with the equilibrium weight. In the present smooth positive-curvature regime, the equilibrium measure is
\[
\mu_\eq^\phi=\frac{1}{V}\frac{\omega^d}{d!}=\frac{1}{V}\Omega .
\]
Since $d\mu=\rho\Omega$, we obtain
\[
w^\phi_{\mu}=\frac{d\mu}{d\mu^\phi_{\eq}}=V\rho .
\]
Consequently,
\[
\|r_k-w^\phi_{\mu}\|_{\mathscr C^1(\M)}=O(k^{-1}),
\]
as claimed.
\end{proof}

\subsection{Laplace transform of linear statistics}

DPPs are known to have tractable Laplace transforms of linear statistics, and \cite{Ber7} has used this Laplace transform to obtain central limit theorems for $\d\mathcal{P}_{(\phi,\mu)}$.
For any continuous $\psi: \M\rightarrow \mathbb{R}$, define the Laplace transform of the linear statistics $\sum_i\psi(X_i)$ as
\begin{equation}\label{eq:Laplace_Transform_LS}
Z_{(\phi,\mu)}[t;\psi] = \E_{(\phi,\mu)}\left[e^{-t\sum_{i=1}^N\psi(X_i)}\right].
\end{equation}
It can also represent the partition function of a Bergman ensemble put in an environment represented by a new reference measure $e^{-t\psi}\d\mu$. The expectation and variance of the linear statistic of the DPP can be expressed in terms of the logarithmic derivative of $Z_{(\phi,\mu)}[t;\psi]$.

\begin{proposition}[\citealp{Ber7}]\label{prop:exp_var}
The function $t\mapsto \log Z_{(\phi,\mu)}[t;\psi]$ is at least twice differentiable with respect to $t$, and satisfies
\begin{equation}
\frac{\d}{\d t}\log Z_{(\phi,\mu)}[t;\psi] = -\E_{(\phi+t\psi,\mu)}\left[\sum_{i=1}^N\psi(X_i)\right] = -\int_\M \psi(x)B_{(\phi+t\psi,\mu)}(x,x)\d\mu(x),
\end{equation}
\begin{align}
\frac{\d^2}{\d t^2}\log Z_{(\phi,\mu)}[t;\psi] &=  \Var_{(\phi+t\psi,\mu)}\left[\sum_{i=1}^N\psi(X_i)\right]\nonumber\\
&=  \frac12\int_{\M^2} (\psi(x)-\psi(y))^2\vert B_{(\phi+t\psi,\mu)}(x,y)\vert_{\phi+t\psi}^2\d\mu^{\otimes 2}(x,y).
\end{align}
\end{proposition}

The next result is a control of the asymptotics of the integral in the expression of the variance in Proposition \ref{prop:exp_var}, when $\phi$ is replaced by $k\phi$ and $k\to\infty$. It is a fundamental result because it shows how the asymptotic variance emerges in the scaling limit in \cite{Ber7}.

\begin{theorem}[\citealp{Ber7}, Theorem 5.8]
  \label{thm:fluctuations_variance_0}
  Let $\M$ be a compact complex manifold of dimension $d$ endowed with a Borel measure $\mu$ associated with a continuous volume form, $L$ be a big line bundle over $\M$ endowed with a $\mathscr{C}^{1,1}$ metric $\phi$, $F$ be a line bundle endowed with a continuous metric with weight $\phi_F$, and $B_{k\phi+\phi_F}$ be the Bergman kernel of $H^0(\M,L^k\otimes F)$. If $f$ is a Lipschitz function with compact support included in the bulk, then 
  \begin{equation}
  \lim_{k\to\infty} \frac{1}{2}\iint_{\M^2} k^{1-d} |B_{k\phi+\phi_F}(x,y)|_{k\phi+\phi_F}^2(f(x)-f(y))^2\d\mu^{\otimes 2}(x,y) = \frac12\Vert \d f\Vert_{\ddc\phi}^2.
  \end{equation}
\end{theorem}
A key ingredient for the proof of our Theorem~\ref{thm:MC_main} will be a generalization of Theorem \ref{thm:fluctuations_variance_0}, taking into account the dependence on $k$ in both $f$ and $\phi_F$, see Proposition~\ref{prop:uniform_variance} below for the precise statement.

\section{CLT for weighted Bergman linear statistics}\label{sec:CLT}

Let $\M$ be a compact complex manifold of dimension $d$, and let $L\to \M$ be a positive holomorphic line bundle endowed with a smooth Hermitian metric $h=e^{-\phi}$ with positive curvature. Let $\mu$ be a reference probability measure which has Radon--Nikodym derivative $\rho\in \mathscr{C}^1(\M)$ with respect to the reference volume form $\Omega=\omega^d/d!$, such that $\rho>0$. Denote by $V=\int_\M\Omega$ the line-bundle volume of $L$. Let $F\to \M$ be a holomorphic line bundle endowed with a $\mathscr{C}^1$ Hermitian metric $e^{-\psi}$. Let $(X_1,\ldots,X_{N_k})$ be the Bergman ensemble for the weighted measure $(k\phi+\psi,\mu)$. For a function $u:\M\to\R$, we define the linear statistic
\[
\Lambda_k(u):=\sum_{i=1}^{N_k} u(X_i).
\]
Let $(u_k)$ be a sequence of functions on $\M$, and consider perturbed weights
\[
\psi_{k,t}:=\psi + tu_k, \qquad t\in I,
\]

\begin{proposition}
\label{prop:uniform_variance}
Let $\M$ be a compact complex manifold of complex dimension $d$, and let
$L\to\M$ be a positive holomorphic line bundle endowed with a smooth Hermitian
metric $e^{-\phi}$ such that $\ddc\phi>0$ on $\M$.
Let $\mu$ be a probability measure with $\mathscr{C}^1$ density, bounded away from $0$, with respect to the volume form $(\ddc\phi)^d/d!$. Let $F\to \M$ be a holomorphic line bundle endowed with a fixed $\mathscr{C}^1$ Hermitian metric $e^{-\psi}$. Let $(u_k)_{k\ge1}\subset \mathscr{C}^1(\M)$ such that $u_k \to u$ in $\mathscr{C}^1(\M)$ and $\sup_{k\ge1}\|u_k\|_{\mathscr{C}^1(\M)}<\infty.$ Let $I\subset \mathbb R$ be a bounded interval and, for $s\in I$, define $\d\mu_{k,s}:=e^{-s u_k}\d\mu.$ For each $k$ and $s$, let $B_{k,s}(x,y):=B_{(k\phi+\psi,\mu_{k,s})}(x,y)$ be the Bergman kernel of the weighted pair $(k\phi+\psi,\mu_{k,s})$.
Then
\[
\sup_{s\in I}
\left|
N_k^{-1+1/d}\Var_{k,\psi,\mu_{k,s}}\big(\Lambda_k(u_k)\big)
-\frac12\mathcal{E}_\phi(u)
\right|
\longrightarrow 0 .
\]
\end{proposition}

\begin{proof}
Set
\[
V_k(s):=N_k^{-1+1/d}\Var_{k,\psi,\mu_{k,s}}\big(\Lambda_k(u_k)\big),\qquad \widetilde{V}_k(s):=k^{1-d}\Var_{k,\psi,\mu_{k,s}}(\Lambda_k(u_k)).
\]
By~\eqref{eq:volume},
\[
N_k^{-1+1/d}=V^{-1+1/d}k^{1-d}(1+O(k^{-1})).
\]
Hence
\[
V_k(s)=V^{-1+1/d}(1+O(k^{-1}))\widetilde V_k(s).
\]
It is therefore enough to prove that
\[
\sup_{s\in I}\left|\widetilde V_k(s)-\frac12\|\d u\|^2_{\ddc\phi}\right|\longrightarrow 0 .
\]
By Proposition~\ref{prop:exp_var},
\begin{equation}\label{4.1}
\widetilde V_k(s)=\frac12k^{1-d}
\iint_{\M\times \M}
(u_k(x)-u_k(y))^2
|B_{k,s}(x,y)|_{k\phi+\psi}^2
\d\mu_{k,s}^{\otimes 2}(x,y).
\end{equation}

We now rewrite the same Bergman ensemble by moving the perturbation from the measure to the auxiliary metric:
\[
(k\phi+\psi,\mu_{k,s})\equiv(k\phi+\psi+s u_k,\mu).
\]
Hence $B_{k,s}$ is also the Bergman kernel attached to the data $(L^k\otimes F,k\phi+\psi+s u_k,\mu).$ Since $\ddc\phi>0$ on all of $\M$, the bulk is all of $\M$. The proof is now the same as the proof of Theorem~5.8 in \cite{Ber7}, and we only indicate the points where uniformity in $s$ enters.

First, because $I$ is bounded and $\sup_k\|u_k\|_{\mathscr{C}^1}<\infty$, the family of
auxiliary weights $\psi_{k,s}$ satisfies
\begin{equation}
\label{4.2}
\sup_{k\ge1}\sup_{s\in I}\|\psi_{k,s}-\psi\|_{L^\infty(\M)}<\infty
\end{equation}
and, for every fixed $R>0$,
\begin{equation}\label{4.3}
\sup_{s\in I}\sup_{x\in \M}\sup_{d(x,y)\le R/\sqrt{k}}
|\psi_{k,s}(x)-\psi_{k,s}(y)|
\longrightarrow 0 .
\end{equation}
Indeed,
\[
|\psi_{k,s}(x)-\psi_{k,s}(y)|
\le |s|\|\d u_k\|_{L^\infty}d(x,y),
\]
and the right-hand side is $O(k^{-1/2})$ uniformly in $k$ and $s$.

Now the only Bergman-kernel inputs used in Berman's proof of Theorem~5.8 are:

\smallskip

\noindent
(i) the local scaling asymptotics on $k^{-1/2}$-balls, namely Theorem~1.1
in \cite{Ber7}, and

\smallskip

\noindent
(ii) the off-diagonal decay estimate, namely Theorem~1.3 in \cite{Ber7}.

\smallskip

By Remark~1.7 in \cite{Ber7}, these results remain valid in the more
general setting $(L^k\otimes F,k\phi+\phi_F)$ with an auxiliary bundle $F$. While Theorem 1.1 and Theorem 1.3 in \cite{Ber7} are stated for a fixed auxiliary weight, inspection of their proofs shows that the constants involved depend only on (i) a uniform $L^\infty$ bound on $\phi_F$, and (ii) the fact that $\phi_F$ is asymptotically constant on $k^{-\frac12}$-balls. In the present setting, both properties hold uniformly in $s\in I$ by \eqref{4.2}–\eqref{4.3}. Therefore, the local scaling asymptotics and off-diagonal estimates (i) and (ii) apply uniformly for the family $\psi_{k,s}$. Therefore the decomposition argument in the proof of Theorem~5.8 may be repeated
verbatim, uniformly in $s$. More precisely, decompose the integral in \eqref{4.1} into the three regions
\[
A_k:=\{d(x,y)\ge1\},\qquad
B_{k,R}:=\{R/\sqrt{k}\le d(x,y)\le1\},\qquad
C_{k,R}:=\{d(x,y)\le R/\sqrt{k}\}.
\]

On $A_k$, the contribution tends to $0$ uniformly in $s$ by the off-diagonal
decay (ii).

On $B_{k,R}$, using the uniform Lipschitz bound on $u_k$ and (ii), the contribution is bounded exactly as in the proof of
Theorem~5.8; hence
\begin{equation}\label{4.4}
\lim_{R\to\infty}\limsup_{k\to\infty}\sup_{s\in I}
\Bigl|\text{contribution of }B_{k,R}\Bigr|=0.
\end{equation}

On $C_{k,R}$, one uses (i) exactly as in Berman's proof.
Since $u_k\to u$ in $\mathscr{C}^1(\M)$, for every fixed $R>0$ we have
\begin{equation}\label{4.5}
\sup_{x\in M}\sup_{|z|\le R}
\left|
\sqrt{k}\bigl(u_k(\exp_x(z/\sqrt{k}))-u_k(x)\bigr)-\d u_x(z)
\right|
\longrightarrow 0,
\end{equation}
and therefore the local contribution converges, uniformly in $s$, to the same
Bargmann--Fock integral as in Theorem~5.8. Consequently,
\begin{equation}\label{4.6}
\lim_{k\to\infty}\sup_{s\in I}
\left|
\text{contribution of }C_{k,R}-A(R)
\right|=0,
\end{equation}
where $A(R)$ is the truncated Gaussian integral appearing in Berman's proof.
Letting $R\to\infty$ and using the final computation in the proof of
Theorem~5.8 gives
\begin{equation}\label{4.7}
A(R)\uparrow \frac12\|\d u\|_{\ddc\phi}^2 .
\end{equation}
Combining \eqref{4.4}, \eqref{4.6}, and \eqref{4.7} yields
\[
\sup_{s\in I}
\left|
\widetilde V_k(s)-\frac12\|\d u\|_{\ddc\phi}^2
\right|
\longrightarrow 0,
\]
which implies the desired result.
\end{proof}

We now state the main result of this section.

\begin{theorem}[CLT for linear statistics]\label{thm:CLT_linear}
Under the assumptions of Proposition~\ref{prop:uniform_variance}, let $(X_1,\ldots,X_{N_k})$ be the Bergman ensemble associated with $(k\phi+\psi,\mu)$. Let $(u_k)$ be a sequence of real-valued functions on $\M$ such that $u_k \to u$ in $\mathscr{C}^1(\M)$ for some $u\in \mathscr{C}^1(\M)$ and $\sup_{k\geq 1}\Vert u_k\Vert_{\mathscr{C}^1(\M)}<\infty$. Then
\[
N_k^{-\frac12+\frac{1}{2d}}\sum_{i=1}^{N_k}
\Bigl(u_k(X_i)-\mathbb{E}[u_k(X_1)]\Bigr)\build{\longrightarrow}{k\to\infty}{(d)}
\mathcal{N} \left(0,\frac12\mathcal{E}_\phi(u)\right).
\]
\end{theorem}

\begin{proof}
Set
\[
Y_k := N_k^{-1/2+1/(2d)}\Big(\Lambda_k(u_k)-\mathbb E_{k,\psi,\mu}[\Lambda_k(u_k)]\Big).
\]
We must prove that $Y_k$ converges in distribution to $\mathcal N\left(0,\frac12\mathcal{E}_\phi(u)\right).$ For $t\in\mathbb R$, set $F_k(t) := \log \mathbb E_{k,\psi,\mu}\left[e^{-tY_k}\right]$. Then, by successive differentiations, we have $F_k(0)=0$, $F_k'(0)=0$, and
\[
F_k''(t)=N_k^{-1+1/d}
\Var_{k,\psi+t N_k^{-1/2+1/(2d)} u_k,\mu} \big(\Lambda_k(u_k)\big).
\]
As we have $N_k^{-1/2+1/(2d)}\to0$, for $t$ in compact sets the parameter $s = t N_k^{-1/2+1/(2d)}$ remains in a bounded interval. Hence Proposition~\ref{prop:uniform_variance} yields
\[
F_k''(t)\longrightarrow \frac12\mathcal{E}_\phi(u)
\]
uniformly for $t$ in compact sets. Finally, since $F_k(0)=F_k'(0)=0$, we have
\[
F_k(t)=\int_0^t (t-\tau)F_k''(\tau)d\tau\longrightarrow\frac{t^2}{4}\mathcal{E}_\phi(u),
\]
and the Laplace transforms of linear statistics converge to the Laplace transform of the desired Gaussian distribution. The result follows by the standard continuity theorem for moment-generating functions.
\end{proof}

\section{Proof of the main results}\label{sec:proofs}

We are now in position to prove the main results of this paper.

\begin{proof}[Proof of Theorem~\ref{thm:MC_main}]
Set $r_k$ as in Lemma~\ref{lem:ratio_C1}, and
\[
f_k:=fr_k,
\quad
f_\eq:=fw^\phi_{\mu},\quad Q_k(f)=\frac{1}{N_k}\sum_{i=1}^{N_k}f_k(X_i)=\sum_{i=1}^{N_k}\frac{f(X_i)}{B_{(k\phi+\psi,\mu)}(X_i,X_i)}. 
\]

We have
\[
\E[Q_k(f)]=\frac{1}{N_k}\E\left[\sum_i f_k(X_i)\right]=\frac{1}{N_k}\int_\M f_k(x)B_{(k\phi+\psi,\mu)}(x,x)\d\mu(x)=\int_\M f(x)\d\mu(x).
\]
Then,
\[
N_k^{-\frac12+\frac{1}{2d}}\left(Q_k(f)-\int_\M f(x)\d\mu(x)\right)=N_k^{-\frac12+\frac{1}{2d}}\sum_{i=1}^{N_k}
\Bigl(f_k(X_i)-\mathbb{E}[f_k(X_1)]\Bigr).
\]
We have $f_k-f_\eq=f(r_k-w^\phi_{\mu}),$ and since multiplication by the fixed $\mathscr{C}^1$ function $f$ is bounded on $\mathscr{C}^1(\M)$, we deduce from Lemma~\ref{lem:ratio_C1} that
\[
\|f_k-f_\eq\|_{\mathscr{C}^1(\M)}=O(k^{-1}).
\]
Thus, the sequence of weights $(f_k)$ satisfies the assumptions of Theorem~\ref{thm:CLT_linear}, and we get
\[
N_k^{-\frac12+\frac{1}{2d}}\sum_{i=1}^{N_k}
\Bigl(f_k(X_i)-\mathbb{E}[f_k(X_1)]\Bigr)
\build{\longrightarrow}{k\to\infty}{(d)}\mathcal{N} \left(0,\frac12\mathcal{E}_\phi(f_\eq)\right),
\]
which concludes the proof.
\end{proof}

\begin{proof}[Proof of Corollary \ref{cor:universality}]
Set $\d\mu_\psi=e^{-\psi}\d\mu$. Since $\mu$ satisfies the assumptions of
Theorem~\ref{thm:MC_main} and $\psi$ is smooth, the measure $\mu_\psi$ satisfies the same assumptions. Apply Theorem~\ref{thm:MC_main} to the weighted measure $(\phi,\mu_\psi)$, with trivial auxiliary bundle, and to the test function $g:=fe^\psi .$ The corresponding Bergman estimator is
\[
\sum_{i=1}^{N_k}\frac{g(X_i)}{B_{(k\phi,\mu_\psi)}(X_i,X_i)}=\sum_{i=1}^{N_k}\frac{f(X_i)e^{\psi(X_i)}}{B_{(k\phi,e^{-\psi}\mu)}(X_i,X_i)}.
\]
It is unbiased for
\[
\int_\M g\d\mu_\psi=\int_\M f e^\psi e^{-\psi}\d\mu=\int_\M f\d\mu,
\]
and we have a CLT with limiting variance $\frac12 \mathcal E_\phi(gw_{\mu_\psi}^\phi).$ Since
\[
w_{\mu_\psi}^\phi=\frac{\d\mu_\psi}{\d\mu_{\eq}^\phi}=e^{-\psi}w_\mu^\phi,
\]
we have $gw_{\mu_\psi}^\phi=fe^\psi e^{-\psi}w_\mu^\phi=fw_\mu^\phi$, and the claimed CLT is proved.
\end{proof}

\section{A detailed example: the Riemann sphere}\label{sec:sphere}

In order to get a concrete idea of how to use our theoretical results, we shall implement them in the simplest compact complex manifold, which is the Riemann sphere. We will first describe the different objects (Bergman kernel, Bergman ensemble) in order to make Theorem~\ref{thm:MC_main} more explicit, then we will provide numerical experiments that illustrate the speed of convergence compared to several other methods.

\subsection{Complex structure and Bergman kernel}

We identify the Riemann sphere with the complex projective line $S^2 \simeq \mathbb{CP}^1.$ In concrete terms, we use the stereographic coordinate
\[
\zeta=\frac{x-iy}{1+z}
\]
on the chart $U_0=S^2\setminus\{(0,0,-1)\}.$ Thus $U_0$ is the sphere without the South pole, which corresponds to the point at infinity in this coordinate. On the second chart $U_1=S^2\setminus\{(0,0,1)\},$ we use
\[
\eta=\frac{x+iy}{1-z}.
\]
On $U_0\cap U_1$, the transition map is
\[
\eta=\frac1\zeta,
\]
which is holomorphic. These two charts therefore endow $S^2$ with its usual complex structure.

We equip $\mathbb{CP}^1$ with the normalized Fubini--Study form. In the coordinate $\zeta$, it is given by
\[
\omega_{\mathrm{FS}}=\frac{i\d\zeta\wedge \d\bar\zeta}{2\pi(1+|\zeta|^2)^2}.
\]
It has total mass one, and we denote the corresponding probability measure by $\d\mathrm{vol}_{S^2}$. Equivalently,
\[
\d\mathrm{vol}_{S^2}(\zeta)=\frac{\d m(\zeta)}{\pi(1+|\zeta|^2)^2},
\]
where
\[
\d m(\zeta)=\frac{i}{2}\d\zeta\wedge \d\bar\zeta
\]
is the Lebesgue measure on $\mathbb C$.

Let $L=\mathcal O(1)$ be the hyperplane line bundle on $\mathbb{CP}^1$, endowed with the Fubini--Study Hermitian metric. In the affine chart $U_0$, this metric has local weight
\[
\phi(\zeta)=\log(1+|\zeta|^2),
\]
and $\ddc\phi=\omega_{\mathrm{FS}}.$ The holomorphic sections of $L^k=\mathcal O(k)$ are represented in the coordinate $\zeta$ by polynomials of degree at most $k$. With respect to the inner product induced by $(k\phi,d\mathrm{vol}_{S^2})$, an orthonormal basis is
\[
s_\ell(\zeta)=\sqrt{k+1}\binom{k}{\ell}^{1/2}\zeta^\ell,\qquad 0\leq \ell\leq k.
\]
In particular, $N_k=\dim H^0(\mathbb{CP}^1,\mathcal O(k))=k+1.$ The associated weighted Bergman kernel is therefore
\[
B_{(k\phi,\d\mathrm{vol}_{S^2})}(\zeta,\xi)=(k+1)\frac{(1+\zeta\bar\xi)^k}{(1+|\zeta|^2)^{k/2}(1+|\xi|^2)^{k/2}}.
\]
In particular, it is constant on the diagonal:
\[
B_{(k\phi,\d\mathrm{vol}_{S^2})}(\zeta,\zeta)=k+1=N_k.
\]

\begin{definition}
The \emph{spherical ensemble} of size $N_k=k+1$ is the Bergman ensemble associated with
\[
(\mathbb{CP}^1,\mathcal O(k),k\phi,\d\mathrm{vol}_{S^2}).
\]
Equivalently, in the stereographic coordinate $\zeta$, it is the determinantal point process on $\mathbb C\cup\{\infty\}$ whose joint density with respect to $\d m^{\otimes N_k}$ is
\[
\frac1{Z_{N_k}}\prod_{1\leq i<j\leq N_k}|\zeta_i-\zeta_j|^2\prod_{i=1}^{N_k}\frac{1}{(1+|\zeta_i|^2)^{N_k+1}}.
\]
\end{definition}

This is the classical spherical ensemble. It is also known to have the same distribution as the eigenvalues of $AB^{-1}$, where $A$ and $B$ are independent standard complex Gaussian matrices of size $N_k\times N_k$ (see \cite{Kri06}).

Since the diagonal Bergman kernel is constant, the estimator of
\[
\int_{S^2} f\d\mathrm{vol}_{S^2}
\]
takes the particularly simple form
\[
\frac1{N_k}\sum_{i=1}^{N_k} f(X_i).
\]
Equivalently, in stereographic coordinates,
\[
\int_{S^2}f(x)\d\mathrm{vol}_{S^2}(x)=\int_{\mathbb C} f\circ\phi_0^{-1}(\zeta)\frac{\d m(\zeta)}{\pi(1+|\zeta|^2)^2},
\]
and the Bergman estimator is
\[
\frac1{N_k}\sum_{i=1}^{N_k} f(X_i).
\]

More generally, if the target measure is
\[
\rho\d\mathrm{vol}_{S^2},\qquad\rho\in C^\infty(S^2),\quad \rho>0,
\]
then the same spherical ensemble can be used with the estimator
\[
\frac1{N_k}\sum_{i=1}^{N_k} f(X_i)\rho(X_i).
\]
The fluctuation scale is unchanged, but the limiting variance is obtained by applying the Dirichlet energy to the reweighted integrand $f\rho$.

\subsection{Numerical experiments}

We shall now proceed to a comparison of our Monte Carlo method, which reduces to using the empirical measure of the spherical ensemble (defined in Definition~\ref{def:spherical_ensemble}) as a Monte Carlo estimator, with a few other estimators. 
We consider a standard Monte Carlo estimator with an i.i.d. uniform sample, a Monte Carlo estimator based on a DPP in $[-1,1]^2$ mapped to the sphere, and a randomized Quasi Monte Carlo estimator.
The latter two are now introduced in more detail, before showing the experimental results.

\subsubsection{Legendre DPP}

The \emph{Jacobi measure} of parameters $\alpha_1,\beta_1,\ldots,\alpha_d,\beta_d>-1$ is the measure on $(-1,1)^d$ given by
\[
\d\mu_{\alpha,\beta}(x_1,\ldots,x_d) = \prod_{j=1}^d(1-x_j)^{\alpha_j}(1+x_j)^{\beta_j}\d x_j.
\]
The corresponding orthonormal polynomials are the so-called multivariate \emph{Jacobi polynomials}; see e.g. \citep{DX}. 
It has been shown by \cite[Theorem 2.2]{BH} that, using a suitable ordering $(p_{k})$ of these multivariate Jacobi polynomials, the projection determinantal point process with kernel $\sum_{k=1}^N p_k(x)p_k(y)$ leads to an diagonal-reweighted estimator with a fast central limit theorem, with a rate that is intermediate between the classical Monte Carlo rate and the rate in Theorem~\ref{thm:MC_main}.
An interesting fact is that integration on $\Sbb^2$ with respect to the uniform measure boils down to an integration on $(-1,1)^2$ with respect to the uniform measure, which is actually the Jacobi measure of parameters $(0,0)$, as explained in the following proposition.

\begin{proposition}\label{prop:Jacobi_int}
For any $f:\Sbb^2\to\R$ measurable and bounded,
\begin{equation}
\int_{\Sbb^2} f(x,y,z)\d\vol_{\Sbb^2}(x,y,z) = \frac{1}{4}\int_{(-1,1)^2}f\circ\Phi(x,y) \d x\d y,
\end{equation}
where $\Phi:[-1,1]^2\to\Sbb^2$ is the function defined by
\[
\Phi(x,y) = (\sqrt{1-x^2}\cos(\pi(y+1)),\sqrt{1-x^2}\sin(\pi(y+1)),x).
\]
\end{proposition}

\begin{proof}
Let $\Phi_1:[0,2\pi]\times[0,\pi]\to\Sbb^2$ be the change of variable from spherical to Cartesian coordinates in the sphere, namely
\[
\Phi_1(\theta,\phi) = (\cos\theta\sin\phi, \sin\theta\sin\phi,\cos\phi).
\]
It is a diffeomorphism from $(0,2\pi)\times(0,\pi)$ onto its image,
whose complement is negligible in $\Sbb^2$ with respect to $\d\vol_{\Sbb^2}$. Moreover,
\[
\Phi_1^*\d\vol_{\Sbb^2}(\theta,\phi) = \frac{1}{4\pi}\sin\phi\d\theta\d\phi.
\]
Hence, for any bounded measurable function $f:\Sbb^2\to\R$,
\begin{align*}
\int_{\Sbb^2}f(x,y,z)\d\vol_{\Sbb^2}(x,y,z) & = \int_{\Phi_1((0,2\pi)\times(0,\pi)} f(x,y,z)\d\vol_{\Sbb^2}(x,y,z)\\
& = \int_{(0,2\pi)\times(0,\pi)} f\circ\Phi_1(\theta,\phi)\sin\phi\frac{\d\theta\d\phi}{4\pi}.
\end{align*}
We also introduce the diffeomorphism
\[
\Phi_2:\left\lbrace\begin{array}{ccc}
(0,2\pi)\times(0,\pi) & \to & (-1,1)^2\\
(\theta,\phi) & \mapsto & (\cos\phi, \frac{\theta}{\pi}-1).
\end{array}\right.
\]
The result follows from the fact that $\Phi=\Phi_1\circ\Phi_2^{-1}$.
\end{proof}

We conclude with a remark that in our case, where $d=2$ and $\alpha=\beta=0$, the multivariate Jacobi polynomials specialize to the Legendre polynomials.
To sample the corresponding DPP, we use the classical algorithm by \cite{HKPV06}, in the specific implementation of the Python library DPPy \citep{GBPV19} for multivariate Jacobi ensembles.

\subsubsection{Randomized spiral points}

Following \cite{RSZ} or \cite{BSSW}, given a fixed parameter $C>0$ and a fixed sample size $N$, the \emph{generalized spiral points} are the points of the sphere with spherical coordinates $(\theta_i,\phi_i)_{1\leq i\leq N}$ defined by an iterative procedure: for any $1\leq i\leq N,$ set $z_i=1-\frac{2i-1}{N}$ and
\[
\theta_i = \arccos z_i, \ \phi_i = C\sqrt{N}\theta_i.
\]
It provides a deterministic low-discrepancy family $\{(x_i,y_i,z_i), 1\leq i\leq N\}$ of points of $\Sbb^2$, which can be randomized through a random (uniform) rotation $R\in\SO(3)$. Although the QMC method using spiral points was studied in the aforementioned papers, we are unaware of any theoretical estimation of the variance of the corresponding randomized QMC.
Yet we expect it to be competitive in our low-dimensional setting.

\subsubsection{Results}

In Figure \ref{fig:samples}, we display samples of all the models we consider. 
Note that all spiral points of the sample are randomized through the same rotation, which makes them look like usual spiral points.
In the case of a Jacobi ensemble, we take a Jacobi ensemble of parameters $(0,0)$ (or equivalently, a Legendre ensemble) on $[-1,1]^2$ mapped onto the sphere through the diffeomorphism $\Phi$ introduced in Proposition \ref{prop:Jacobi_int}. 
It corresponds to the method of \cite{BH}.
As expected, we see that a cluster appears in the image of the boundary of the square $[-1,1]^2$.
\begin{figure}[!h]
    \centering
    \subfloat[i.i.d. uniform]{\includegraphics[width=0.45\textwidth]{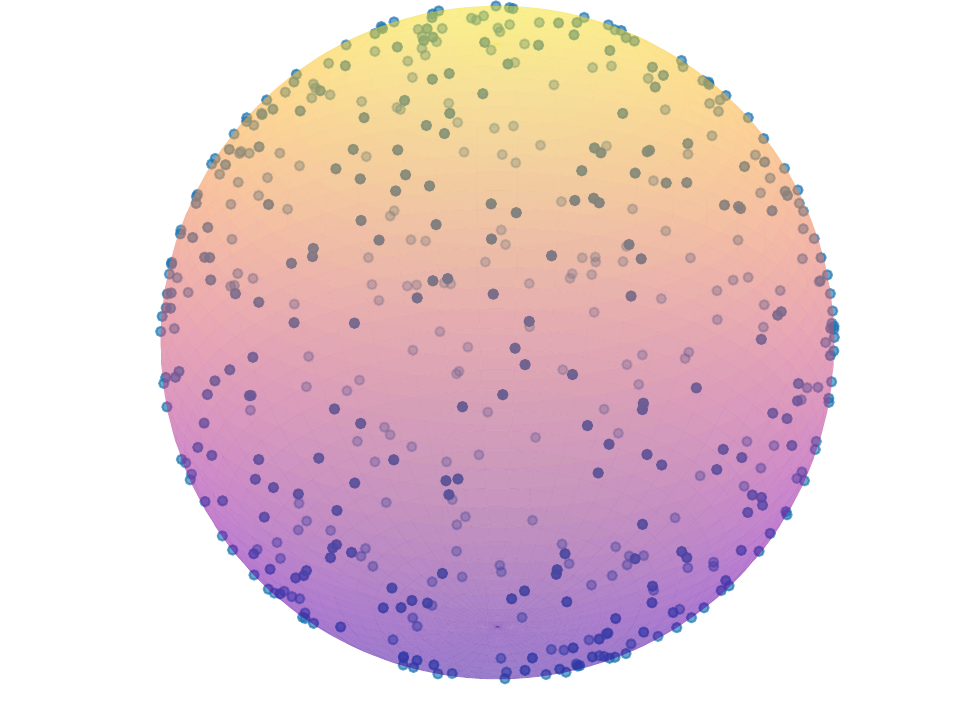}}
    \subfloat[randomized spiral]{\includegraphics[width=0.45\textwidth]{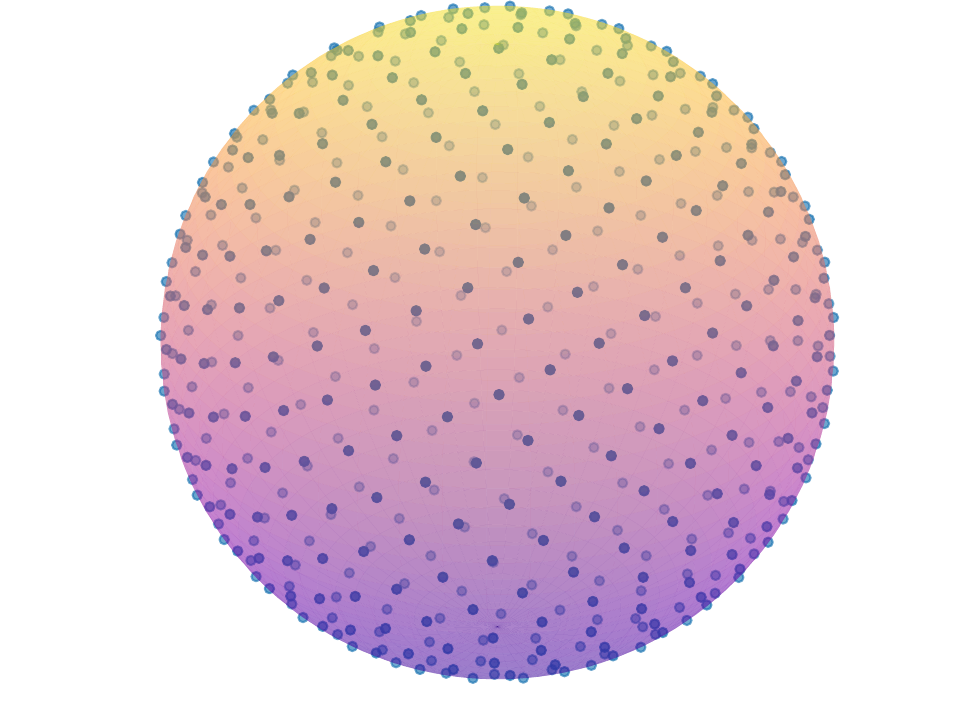}}\\
    \subfloat[spherical (or Bergman) ensemble]{\includegraphics[width=0.45\textwidth]{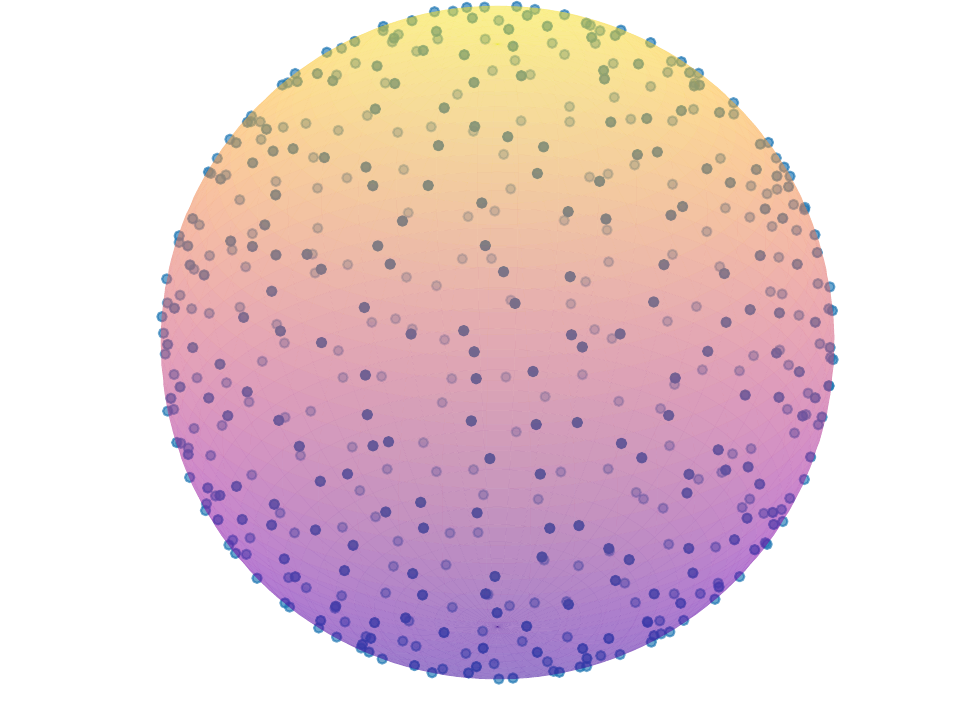}}
    \subfloat[Jacobi ensemble]{\includegraphics[width=0.45\textwidth]{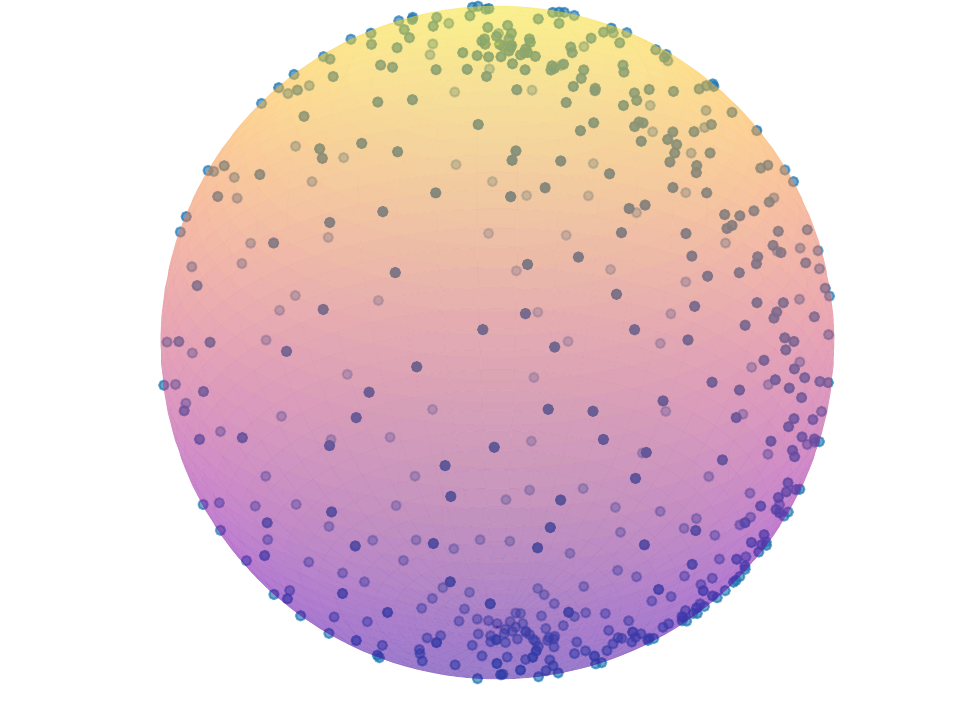}}
    \caption{Independent samples of size $N=500$ from four distributions on the sphere.}
    \label{fig:samples}
\end{figure}

Figure~\ref{fig:integrals} displays the logarithm of each sample variance as a function of $\log N$, across 200 independent repetitions for each model and each $N$. 
Keeping in mind that the variance should be proportional to $N^\alpha$ for various values of $\alpha$ depending on the model, all plots are supposed to be linear.
\begin{figure}[!h]
    \centering
    \subfloat[$f_1$]{
      \includegraphics[scale=0.45]{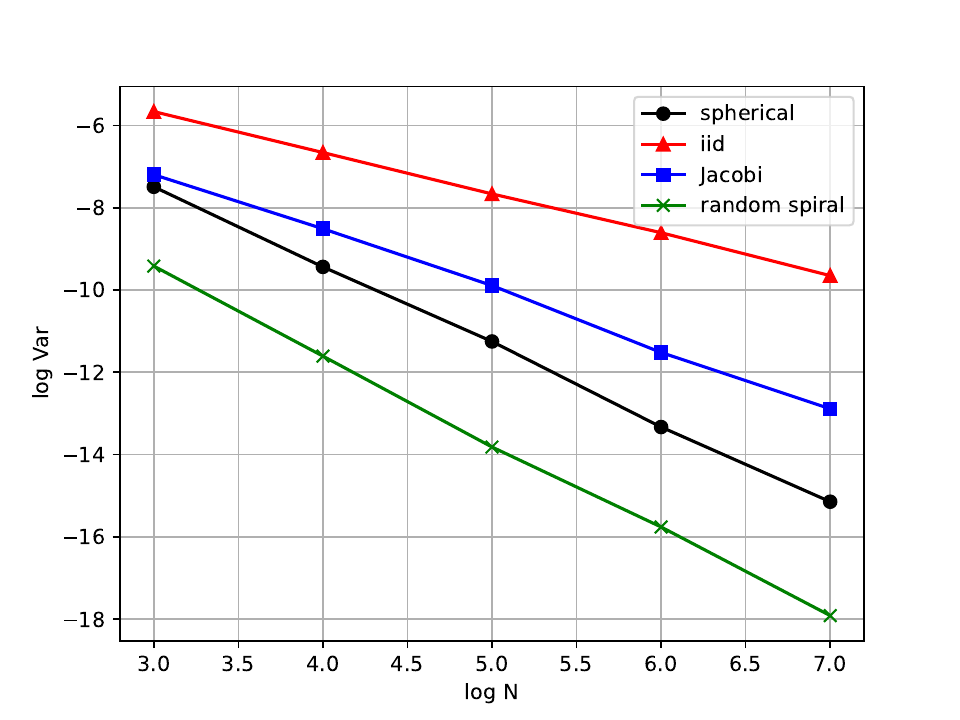}
    }
    \subfloat[$f_2$]{
      \includegraphics[scale=0.45]{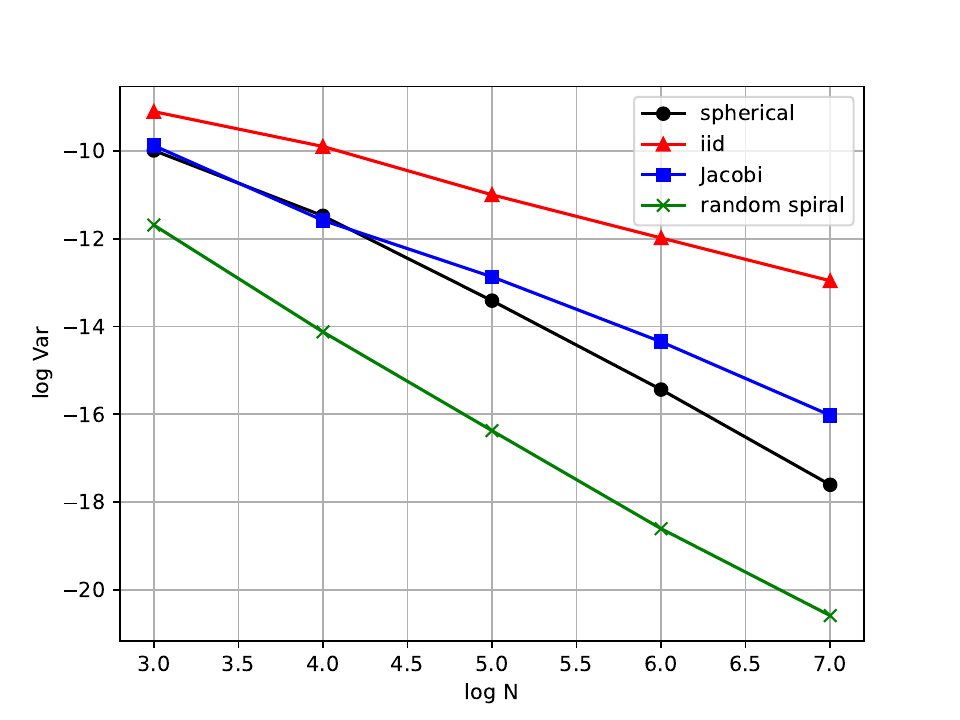}
    }
    \caption{$\log$ variance of four integral estimators wrt. the number $N$ of quadrature nodes.}
    \label{fig:integrals}
\end{figure}
We consider two functions $f_1,f_2:\R^3\to\R$ that we restrict to $\Sbb^2$, and we choose them with support in
\[
\Sbb_+^2=\{(x,y,z)\in\Sbb^2: z\geq 0\}
\]
in order to avoid numerical errors that could happen in stereographic coordinates for points which are close to the South pole (corresponding to the point at infinity). We take
\[
f_1(x,y,z) = z^2\mathbf{1}_{z\geq 0}, \quad f_2(x,y,z) = \vert x\vert^{3/2}yz^2\mathbf{1}_{z\geq 0}.
\]
Both functions are $\mathscr{C}^1$ on $\Sbb_+^2$, but $f_1$ is actually smooth on the open hemisphere $\{z>0\}$. Besides, $f_2$ is supported in the image of the open square $(-1,1)^2$ and therefore satisfies the assumptions of \cite{BH}, whereas $f_1$ is nonzero on the image of the boundary of the square. Both functions also naturally satisfy the assumptions of the CLT for the i.i.d. Monte Carlo method. In both cases, the variances of the estimators have the same rankings, and the slopes are quite close to their theoretical values. It is interesting to see that for $f_2$, on low values of $N$, the estimator for the Jacobi ensemble has a slightly lower variance than the spherical ensemble, although the former decays more slowly when $N$ grows. It does not happen for $f_1$, which is not surprising because it is an edge case for the method by \cite{BH}. We also remark that the randomized spiral points seem to provide an overall better performing estimator than all other methods, with a similar slope to the spherical ensemble, although there is no theoretical result to support that.

\section{Conclusion and perspectives}
\label{sec:discussion}

Building on Berman's seminal work that led to the central limit theorem in \citep{Ber7}, we showed that Bergman ensembles can lead to fast Monte Carlo integration on compact complex manifolds, just like multivariate orthogonal polynomial ensembles (OPEs) yield fast quadrature on compacts of the Euclidean space \citep{BH}. 
The take-home message is that, like OPEs, Bergman ensembles come up with a \emph{fast} central limit theorem for Monte Carlo integration that is also \emph{universal}, in the sense that the asymptotic variance of the central limit theorem is invariant to a suitable change of the reference measure and the kernel.
Unlike OPEs, however, the dimension in the rate of convergence is now the \emph{complex} dimension of the manifold, which has the important consequence that Bergman ensembles outperform multivariate OPEs for integration in Euclidean spaces of even real dimension.
Actually, the error rate for Bergman ensembles matches the optimal worst-case rate by \cite{Bak} for functions of class $\mathscr{C}^1$.

Future work includes the following tasks, roughly ranked by increasing difficulty. 
We shall first investigate the influence of the smoothness of the integrand on the error decay, in line with \citep{BeBaCh19,BeBaCh20}.
Then, if we want to make DPPs on arbitrary complex manifolds practical, we need to circumvent the fact that the Bergman kernel is usually only available in the form of asymptotic estimates. 
We should thus investigate the statistical effect of working with an approximate DPP built using these estimates. 
This will be facilitated by recent results \citep{JaKe26Sub} on how to transfer results on the variance of linear statistics from one DPP to another, in particular from an intractable DPP such as the Bergman ensemble on an arbitrary complex manifold and a tractable approximation, obtained using mathematical estimates or a suitable discretization procedure.
Alternately, it might be possible to use importance sampling with a proposal DPP for which we know how to numerically evaluate the Bergman kernel, for instance, using the Kodaira embedding of a K\"ahler manifold in a complex projective space of higher dimension, where the Bergman kernel is explicit. 
Another source of inspiration is \cite{EtAr26Sub}, who lift repulsive configurations such as DPPs on small-dimensional spaces through fibrations and quasi-Monte Carlo constructions, preserving the statistical benefits of the original DPP. 

Finally, while compact complex manifolds already include important practical settings, such as Bayesian quantum tomography, extending the results to more general manifolds, e.g. with a boundary, would further broaden the applicability of DPP-based Monte Carlo integration.

\subsection*{Acknowledgements}

TL thanks Yohann Le Floch and we both thank Martin Rouault for insightful discussions. 
We acknowledge support from ERC grant BLACKJACK (ERC-2019-STG-851866) and ANR AI chair
BACCARAT (ANR-20-CHIA-0002). 

\appendix

\section{A brief reminder of complex geometry}\label{appendix}

\subsection{Basic definitions}

We start with the notion of complex manifold.

\begin{definition}\label{def:cpx_1}
A \emph{complex manifold} of dimension $d$ is a topological space $\M$ endowed with a family $(U_i,\varphi_i)_{i\in I}$ of open subsets $U_i\subset \M$ and homeomorphisms $\varphi_i:U_i\to\varphi_i(U_i)\subset\C^d$ such that, if $U_i\cap U_j\neq \varnothing$,
\[
\varphi_i\circ\varphi_j^{-1}:\varphi_j(U_i\cap U_j) \to \varphi_i(U_i\cap U_j)
\]
is a biholomorphism\footnote{That is, a holomorphic bijection whose inverse is also holomorphic.} between open subsets of $\C^d$. The open subsets $U_i$ are called \emph{charts}, and the maps $\varphi_i$ \emph{local coordinates}.
\end{definition}

The fundamental idea of complex manifolds is that the compatibility of charts and coordinates enables the use of the usual tools of complex analysis on $\C^d$. 
For instance, a function $f:\M\to\C$ is holomorphic (resp. $\mathscr{C}^s$) if for any $i\in I$ the function $f\circ\varphi_i^{-1}:\varphi_i(U_i)\to \C$ is holomorphic (resp. $\mathscr{C}^s$).

\begin{definition}\label{def:cpx_2}
Let $\M$ be a complex manifold. A \emph{holomorphic vector bundle of rank} $r$ over $\M$ is a complex manifold $E$ endowed with a holomorphic surjective map $\pi:E\to \M$ such that, for any $x\in \M$, the \emph{fiber} $E_x=\pi^{-1}(x)$ is an $r$-dimensional vector space over $\C$. 
A holomorphic vector bundle of rank $r$ is \emph{locally trivial} if there exist an open covering $(V_j)_{j\in J}$ of $\M$ and biholomorphic maps $\psi_j:\pi^{-1}(V_j)\to V_j\times\C^r$ such that the diagram
\[
\begin{tikzcd}
\pi^{-1}(V_j) \ar[rr, "\psi_j"] \ar[dr, "\pi"'] & & V_j\times\C^r \ar[dl, "\mathrm{pr}_1"]\\
& V_j &
\end{tikzcd}
\]
commutes, and such that the restriction of $\psi_j$ to $E_x$ is a $\C$-linear map for all $x\in \M$. 
The maps $\psi_j$ are called \emph{trivialization functions}.
\end{definition}
In other words, in a locally trivial vector bundle, on each $V_j$, $\pi$ is akin to the canonical projection of $V_j\times \mathbb{C}^r$ onto $V_j$. Vector bundles are usually denoted as $E\rightarrow \M$.

\begin{definition}\label{def:cpx_3}
Let $E$ be a holomorphic vector bundle over a complex manifold $\M$.
A \emph{local section} of $E$ is a continuous map $s:U\to E$, for {$U\subset \M$ open,} such that $\pi\circ s = \Id_\M$ on $U$. A local section defined on $\M$ is called a \emph{global section}.
\end{definition}

Heuristically, taking a section of the vector bundle is equivalent to taking a continuous family of vectors indexed by an open subset of $\M$, that is, a vector field on $\M$. We denote by $\mathscr{C}^s(\M,E)$ (resp. $H^0(\M,E)$) the space of $\mathscr{C}^s$ (resp. holomorphic) sections of $E$, for any $0\leq s\leq\infty$.

\subsection{Operations on vector bundles}

As a holomorphic vector bundle $E$ over a complex manifold $\M$ induces a complex vector space on each fiber, one can leverage this linear algebraic structure in two ways: to perform algebraic transformations (direct sum, tensor product, wedge product) or to enrich the line bundle. The idea is that any such operation can be done on a fiber, in a way which is explicit whenever one works in a local trivialization. 

{\bf Hermitian metrics.} If we endow $E$ with a family of Hermitian inner products $h_x:E_x\times E_x\to\C$, the vector bundle is said to be \emph{Hermitian}, and $h$ is called a \emph{Hermitian metric}. 
The regularity (e.g. continuous, differentiable, smooth) of the metric, by convention, will be the regularity of the map $x\mapsto h_x$. This regularity can be made more explicit by using the notion of local weight: let $U\subset \M$ be an open subset where $E$ can be trivialized. There exists a section $e_U:U\to E$ such that $e_U(x)\neq 0$ for all $x\in U$, called \emph{local frame}, such that for all $s\in H^0(\M,E)$, there exists $f:U\to\C$ that satisfies
\[
s(x) = f(x)e_U(x),\ \forall x\in U.
\]
In this case, the Hermitian metric $h$ reads
\[
h_x(s_1(x),s_2(x))=f_1(x)\overline{f_2}(x) h_x(e_U(x),e_U(x)) = f_1(x)\overline{f_2}(x) \e^{-\phi(x)},
\]
where $\phi:U\to\R$ is defined by $\phi(x)=-\log h_x(e_U(x),e_U(x))$ and is called the \emph{local weight} of $h$. The regularity of the metric $h$ is then equivalent to the regularity of the weight $\phi$ as a function on (an open subset of) the complex manifold $\M$. We will often identify the metric $h$ with its local weight for the sake of simplicity. We will also sometimes denote by $\vert\cdot\vert_{\phi}$ the norm induced by the metric $h$ with local weight $\phi$.

{\bf Pullback bundle.} If $E\to\M_2$ is a vector bundle with projection $\pi$ and $f:\M_1\to\M_2$ is a holomorphic function, then we define the \emph{pullback bundle} $f^*E\to\M_1$ as
\[
f^*E=\{(m,v)\in\M\times E: f(m)=\pi(v)\}\subset\M\times E,
\]
with projection $\widetilde{\pi}(m,v)=m$. 

{\bf Dual bundle.} If $E\to\M$ is a vector bundle, one can define its \emph{dual bundle} $E^*\to\M$ as the vector bundle whose fibers are the dual fibers of $E$: $(E^*)_x=(E_x)^*$. If $E$ is endowed with a Hermitian metric $h$, we shall denote $\overline{E}$ its dual. For any section $s_1$ of $E$, there is a unique\footnote{This is a well-known consequence of the Riesz representation theorem.} section $\overline{s_1}$ of $\overline{E}$ such that
\begin{equation}\label{eq:contraction}
(\overline{s_1}(x), s_2(x)) = h_x(s_2(x),s_1(x)), \quad \forall s_2\in H^0(\M,E), \forall x\in \M,
\end{equation}
where parentheses in the left-hand side denote the duality pairing. 

{\bf Tensor products of bundles.} If $E_1\to \M$ and $E_2\to \M$ are two vector bundles over the same manifold, their \emph{tensor product} is the vector bundle $E_1\otimes E_2\to \M$ whose fibers are defined by tensor products of the fibers of $E_1$ and $E_2$. 

In a similar fashion, if $E_1\to \M_1$ and $E_2\to \M_2$ are two vector bundles over two separate manifolds, their \emph{external tensor product} is the vector bundle
\[
E_1\boxtimes E_2 = \mathrm{pr}_1^*E_1\otimes\mathrm{pr}_2^*E_2\to \M_1\times \M_2,
\]
where we denoted by $\mathrm{pr}_1:\M_1\times \M_2\to \M_1$ and $\mathrm{pr}_2:\M_1\times \M_2\to \M_2$ the standard projections. It must be distinguished from the previously defined tensor product of two bundles over the same manifold $\M$.

\subsection{Differential forms and integration on complex manifolds}

Any complex manifold $\M$ of dimension $d$ is also a smooth real manifold of dimension $2d$, \emph{i.e.} it can be locally modelled on $\R^{2d}$ by the natural identification $\C\cong\R^2$. The \emph{tangent bundle} $T\M$ of $\M$ is the smooth vector bundle\footnote{Replace holomorphic by smooth and 1-dimensional complex vector space by $2d$-dimensional real vector space in Definition \ref{def:cpx_1}.} of rank $2d$ such that for any $x\in \M$, the fibre $T_x\M$ is the tangent space of $\M$ at $x$. There is a morphism of real vector bundles $J:T\M\to T\M$ such that $J^2=-\Id_{T\M}$, defined on each fibre by $J_x:T_x\M\to T_x\M$, which satisfies the Cauchy--Riemann equations on open subsets $U\subset \M$
\[
\d f_x(J_xv)=\i \times \d f_x(v),\ \forall v\in T_x\M,\ \forall f\in\mathscr{O}(U),\ \forall x\in U.
\]
The cotangent bundle $T^*\M=\Hom_\R(T\M,\C)$ splits into $\Hom_\C(T\M,\C)\oplus\Hom_\C(T\M,\bar\C)$ of $\C$-linear and $\C$-antilinear maps, where $T\M$ is endowed with the complex structure induced by the morphism $J$. We denote respectively by $T^{*(1,0)}\M$ and $T^{*(0,1)}\M$ the subspaces of this decomposition. If $(z_1,\ldots,z_d)$ is a local holomorphic coordinate system in an open subset $U\subset \M$ (for instance, in the atlas $(U_i,\varphi_i)$, it means that we note $\varphi_i(x)=(z_1,\ldots,z_d)\in\C^d$ for any $x\in U_i$), and if we set $z_j=x_j+\i y_j$, then $(x_j,y_j)$ is a local smooth coordinate system of $\M$ as a real manifold, and $(\d z_1,\ldots,\d z_d)$ is a local frame of $T^{*(1,0)}\M$, where $\d z_j=\d x_j+\i \d y_j$. Analogously, $(\d\bar{z}_1,\ldots,\d\bar{z}_d)$ is a local frame of $T^{*(0,1)}\M$, where $\d\bar{z}_j=\d x_j-\i \d y_j$.

\begin{definition}
The \emph{bundle of} $(p,q)$-\emph{forms} on a complex manifold $\M$ is the vector bundle $\Lambda^{p,q}(T^*\M)=\Lambda^p(T^{*(1,0)}\M)\otimes\Lambda^q(T^{*(0,1)}\M)$. We denote by $\Omega^{p,q}(\M)$ the subspace of smooth $(p,q)$-forms on $\M$.
\end{definition}
Any $(p,q)$-form $\omega$ on $\M$ can be expressed as follows in local coordinates:
\begin{equation}\label{eq:local_form}
\omega(x) = \sum_{\substack{1\leq i_1<\cdots<i_p\leq d\\ 1\leq j_1<\cdots<j_q\leq d}} u_{i_1,\ldots,i_p,j_1,\ldots,j_q}(x) \d z_{i_1}\wedge\cdots\wedge \d z_{i_p}\wedge \d\bar{z}_{j_1}\wedge\cdots\wedge \d\bar{z}_{j_q},
\end{equation}
and in particular any $(0,0)$-form is simply a function $f:\M\to\C$. There are two differential operators of interest, called the \emph{Dolbeault operators}:
\[
\partial:\left\lbrace\begin{array}{ccc}
\Omega^{(p,q)}(\M) & \to & \Omega^{(p+1,q)}(\M)\\
f(x)\omega(x) & \mapsto & \sum_i\frac{\partial f(x)}{\partial z_i}\d z_i\wedge\omega(x)
\end{array}\right.,\
\bar\partial:\left\lbrace\begin{array}{ccc}
\Omega^{(p,q)}(\M) & \to & \Omega^{(p+1,q)}(\M)\\
f(x)\omega(x) & \mapsto & \sum_i\frac{\partial f(x)}{\partial \bar{z}_i}\d \bar{z}_i\wedge\omega(x)
\end{array}\right.,
\]
and they can be turned into operators $\d=\partial+\bar\partial$ and $\d^c=\frac{1}{4\i\pi}(\partial-\bar\partial)$, where $\d$ coincides with the exterior derivative.

A \emph{volume form} on $\M$ is a volume form in the differential sense, therefore a nonvanishing section of $\Lambda^{2d}(T^*\M)\simeq\Lambda^{d,d}(T^*\M)$, and it can be written locally on $U\subset \M$ as
\[
\omega(x)=u(x)\d z_1\wedge\cdots\wedge \d z_d\wedge \d\bar{z}_1\wedge\cdots\wedge \d\bar{z}_d,
\]
where $u:\M\to\C$ is a function that does not vanish. The volume form is continuous (resp. smooth, holomorphic) if and only if $u$ is continuous (resp. smooth, holomorphic).

Any positive real volume form on $\M$ can be identified to a Borel measure $\d\mu$ on $\M$ by setting $\int_U \d\mu = \int_U \omega$ for any Borel set $U\subset \M$.

\bibliographystyle{plainnat}
\bibliography{MC_Bergman,stats}

\end{document}